\documentclass[12pt]{amsart}

\usepackage{amssymb,a4wide,graphicx,hyperref}

\newtheorem{theorem}{Theorem}
\newtheorem{corollary}{Corollary}
\newtheorem{lemma}{Lemma}[section]
\newtheorem{proposition}[lemma]{Proposition}
\theoremstyle{definition}
\newtheorem{definition}{Definition}
\theoremstyle{remark}
\newtheorem{remark}[lemma]{Remark}
\newtheorem{example}[lemma]{Example}

\numberwithin{equation}{section}

\title{Rational self-affine tiles}

\dedicatory{Dedicated to Professor Shigeki Akiyama on the occasion of his $50^{th}$ birthday}

\author[W.~Steiner]{Wolfgang Steiner}
\address{LIAFA, CNRS UMR 7089, Universit\'e Paris Diderot -- Paris 7, Case 7014, 75205 Paris Cedex 13, FRANCE}
\email{steiner@liafa.univ-paris-diderot.fr}
\author[J. M. Thuswaldner]{J\"org M. Thuswaldner}
\address{Chair of Mathematics and Statistics, University of Leoben, A-8700 Leoben, AUSTRIA}
\email{joerg.thuswaldner@unileoben.ac.at}
\thanks{This research was supported by the Austrian Science Foundation (FWF), project S9610, which is part of the national research network FWF--S96 ``Analytic combinatorics and probabilistic number theory'', and by the Amad\'ee grant FR 16/2010 -- PHC Amadeus 2011 ``From fractals to numeration''. }

\date{March 3, 2012}

\keywords{Self-affine tile, tiling, shift radix system}

\subjclass[2010]{52C22, 11A63, 28A80}

\begin{document}
\begin{abstract}
An integral self-affine tile is the solution of a set equation $\mathbf{A} \mathcal{T} = \bigcup_{d \in \mathcal{D}} (\mathcal{T} + d)$, where $\mathbf{A}$ is an $n \times n$ integer matrix and $\mathcal{D}$ is a finite subset of $\mathbb{Z}^n$.
In the recent decades, these objects and the induced tilings have been studied systematically.
We extend this theory to matrices $\mathbf{A} \in \mathbb{Q}^{n \times n}$.
We define rational self-affine tiles as compact subsets of the open subring $\mathbb{R}^n\times \prod_\mathfrak{p} K_\mathfrak{p}$ of the ad\`ele ring $\mathbb{A}_K$, where the factors of the (finite) product are certain $\mathfrak{p}$-adic completions of a number field $K$ that is defined in terms of the characteristic polynomial of~$\mathbf{A}$.
Employing methods from classical algebraic number theory, Fourier analysis in number fields, and results on zero sets of transfer operators, we establish a general tiling theorem for these tiles.

We also associate a second kind of tiles with a rational matrix.
These tiles are defined as the intersection of a (translation of a) rational self-affine tile with $\mathbb{R}^n \times \prod_\mathfrak{p} \{0\} \simeq \mathbb{R}^n$.
Although these intersection tiles have a complicated structure and are no longer self-affine, we are able to prove a tiling theorem for these tiles as well.
For particular choices of the digit set $\mathcal{D}$, intersection tiles are instances of tiles defined in terms of shift radix systems and canonical number systems.
This enables us to gain new results for tilings associated with numeration systems.
\end{abstract}

\maketitle

\section{Introduction}

Let $\mathbf{A} \in \mathbb{R}^{n\times n}$ be a real matrix which is \emph{expanding}, i.e., all its eigenvalues are outside the unit circle, and let $\mathcal{D} \subset\mathbb{R}^n$ be a finite ``digit set''.
Then, according to the theory of iterated function systems (see e.g.\ Hutchinson~\cite{Hutchinson:81}), there is a unique non-empty compact set $\mathcal{T} = \mathcal{T}(\mathbf{A},\mathcal{D}) \subset \mathbb{R}^n$ satisfying the set equation
\[
\mathbf{A}\, \mathcal{T} = \bigcup_{d \in \mathcal{D}} (\mathcal{T} + d)\,.
\]
If $\mathcal{T}$ has positive Lebesgue measure, it is called a \emph{self-affine tile}. The investigation of these objects began with the work of Thurston~\cite{Thurston:89} and Kenyon~\cite{Kenyon:92}.
The foundations of a systematic theory of self-affine tiles were provided by Gr\"ochenig and Haas~\cite{Groechenig-Haas:94} as well as Lagarias and Wang \cite{Lagarias-Wang:96a,Lagarias-Wang:96b,Lagarias-Wang:96c,Lagarias-Wang:97} in the middle of the 1990s.
Up to now, various properties of self-affine tiles have been investigated. For instance, there are results on geometric \cite{Kenyon-Li-Strichartz-Wang:99,Strichartz-Wang:98,Duvall-Keesling-Vince:00} and topological \cite{Kirat-Lau:00,Bandt-Wang:01,Akiyama-Thuswaldner:04,Leung-Lau:07} aspects as well as characterizations of ``digit sets''~$\mathcal{D}$ that provide a self-affine tile for a given matrix~$\mathbf{A}$; see for instance \cite{Odlyzko:78,Lagarias-Wang:96a,Lau-Rao:03,Lai-Lau-Rao:13}.

An important feature of self-affine tiles are their remarkable tiling properties.
Particularly nice tiling properties come up if $\mathbf{A}$ is an integer matrix and $\mathcal{D}$ is a subset of~$\mathbb{Z}^n$. In this case, the tile $\mathcal{T}(\mathbf{A},\mathcal{D})$ is called \emph{integral self-affine tile}.
The starting point for our paper is the following result of Lagarias and Wang \cite{Lagarias-Wang:97} on tiling properties of such tiles.

Let $\mathbf{A}$ be an expanding integer matrix with irreducible characteristic polynomial and $\mathcal{D}$ be a complete set of coset representatives of the group $\mathbb{Z}^n / \mathbf{A} \mathbb{Z}^n$.
Denote by $\mathbb{Z}\langle \mathbf{A}, \mathcal{D}\rangle$ the smallest $\mathbf{A}$-invariant sublattice of $\mathbb{Z}^n$ containing the difference set $\mathcal{D} - \mathcal{D}$.
Then $\mathcal{T} = \mathcal{T}(\mathbf{A},\mathcal{D})$ induces a lattice tiling of the space $\mathbb{R}^n$ with respect to~$\mathbb{Z}\langle \mathbf{A}, \mathcal{D}\rangle$; see \cite[Corollary~6.2 and Lemma~2.1]{Lagarias-Wang:97}.
In particular, this means that
\[
\bigcup_{\mathbf{z} \in \mathbb{Z}\langle \mathbf{A},\mathcal{D}\rangle} (\mathcal{T}+\mathbf{z}) = \mathbb{R}^n \quad \hbox{with} \quad \mu\big((\mathcal{T}+\mathbf{z}) \cap (\mathcal{T}+\mathbf{z}')\big) = 0 \ \hbox{for all}\ \mathbf{z}, \mathbf{z}'\in \mathbb{Z}\langle \mathbf{A},\mathcal{D}\rangle,\ \mathbf{z} \not= \mathbf{z}',
\]
where $\mu$ denotes the ($n$-dimensional) Lebesgue measure.

We mention that tiling questions are of interest also in more general contexts. For instance, the tiling problem is of great interest for self-affine tiles that are defined as solutions of graph-directed systems (see for instance \cite{Kenyon-Vershik:98,Grochenig-Haas-Raugi:99,Lagarias-Wang:03,Ito-Rao:06,Kenyon-Solomyak:10}). However, the results are less complete here than in the above setting. Indeed, the quest for tiling theorems in the graph-directed case includes the {\it Pisot conjecture} which states that each Rauzy fractal associated with an irreducible unit Pisot substitution induces a tiling; see e.g. \cite{Barge-Kwapisz:06,Ito-Rao:06}.

The first aim of the present paper is to extend the tiling theorem of Lagarias and Wang in another direction. We shall define self-affine tiles (and tilings) associated with an expanding matrix $\mathbf{A} \in \mathbb{Q}^{n\times n}$ with irreducible characteristic polynomial, and develop a tiling theory for these tiles.

The first kind of tiles we are dealing with are self-affine tiles defined in spaces of the shape $\mathbb{R}^n\times \prod_\mathfrak{p} K_\mathfrak{p}$ where the factors of the (finite) product are certain $\mathfrak{p}$-adic completions of a number field $K$ that is defined in terms of the characteristic polynomial of~$\mathbf{A}$.
We call these tiles \emph{rational self-affine tiles} and establish fundamental properties of these objects in Theorem~\ref{basicProperties}.
Using characters of the ad\`ele ring $\mathbb{A}_K$ of $K$ and Fourier analysis on the locally compact Abelian group $\mathbb{R}^n\times \prod_\mathfrak{p} K_\mathfrak{p}$, we establish a Fourier analytic tiling criterion in the spirit of \cite[Proposition~5.3]{Groechenig-Haas:94} for these tiles; see Proposition~\ref{p:53}. This criterion is then used to establish a tiling theorem (stated as Theorem~\ref{newtilingtheorem}) for rational self-affine tiles in the flavor of the one by Lagarias and Wang mentioned above.
To this matter, as in~\cite{Lagarias-Wang:97}, we have to derive properties of the zero set of eigenfunctions of a certain transfer operator related to the tile under consideration. To achieve this, we use methods from classical algebraic number theory. One of the difficulties in the proof comes from the fact that Lagarias and Wang use a result on zero sets of transfer operators due to Cerveau, Conze, and Raugi~\cite{Cerveau-Conze-Raugi:96}. As this result seems to have no analogue in spaces containing $\mathfrak{p}$-adic factors, we have to adapt our setting to make it applicable in its original form.

It turns out that it is more natural to define a rational self-affine tile in terms of an expanding algebraic number~$\alpha$ rather than an expanding matrix $\mathbf{A} \in \mathbb{Q}^{n\times n}$. It will become apparent in Section~\ref{sec:basic-defin-main} that both formulations lead to the same objects. We mention here that in the context of Rauzy fractals, tilings with $\mathfrak{p}$-adic factors have been investigated in \cite{Siegel:03,ABBS:08}. However, in this setting up to now no general tiling theorems are known.

Interestingly, it is possible to associate a second kind of tiles with a rational matrix. These tiles, which turn out to be ``slices'' of a single rational self-affine tile (see Proposition~\ref{p:slices}) are defined as the intersection of a rational self-affine tile (as well as its translates) with $\mathbb{R}^n\times \prod_\mathfrak{p} \{0\}$. As this space is obviously isomorphic to~$\mathbb{R}^n$, these intersections can be regarded as subsets of the Euclidean space. In general, intersections of fractals with subspaces are hard to handle. In our context, the self-affine structure is lost and the tiles even cannot be described by a graph-directed system. Nevertheless, we are able to show that these tiles give rise to tilings of~$\mathbb{R}^n$ (Theorem~\ref{srstiling}). In proving this, we also show that the boundary of these tiles has zero $n$-dimensional Lebesgue measure. Note that the tiles forming such a tiling may have infinitely many different shapes and not each of them has to be equal to the closure of its interior (see Example~\ref{ex:1}). Nevertheless, we exhibit certain ``almost periodicity'' properties of these tilings for a large class of digit sets (Theorem~\ref{t:24}).

These ``intersective'' tilings are of special interest as they are related to so-called \emph{shift radix systems} (SRS for short), which form common generalizations of several kinds of numeration systems like canonical number systems and beta numeration; see \cite{Akiyama-Borbeli-Brunotte-Pethoe-Thuswaldner:05}.
Shift radix systems are dynamical systems depending on a parameter $\mathbf{r} \in \mathbb{R}^n$, defined by $\tau_\mathbf{r}:\, \mathbb{Z}^n \to \mathbb{Z}^n$,  $\mathbf{z} = (z_1, \ldots, z_n) \mapsto (z_2, \ldots, z_n,-\lfloor \mathbf{r z} \rfloor)$, where $\mathbf{r z}$ is the scalar product.
We can write $\tau_\mathbf{r}(\mathbf{z}) = \mathbf{M}_\mathbf{r}\, \mathbf{z} + (0, \ldots, 0, \mathbf{r z} - \lfloor \mathbf{r z} \rfloor)$, where $\mathbf{M}_\mathbf{r}$ is the companion matrix to the vector~$\mathbf{r}$.
SRS exhibit interesting properties when the spectral radius $\varrho(\mathbf{M}_\mathbf{r})$ of $\mathbf{M}_\mathbf{r}$ is less than~$1$, i.e., if $\mathbf{r}$ is contained in the so-called Schur-Cohn region $\mathcal{E}_n = \{\mathbf{r}\in\mathbb{R}^n:\, \varrho(\mathbf{M}_\mathbf{r}) < 1\}$; see \cite{Schur:18}. Recently, \emph{SRS tiles} associated with $\mathbf{r} \in \mathcal{E}_n$ were defined in \cite{BSSST:11} by the Hausdorff limit
\begin{equation}\label{srstiledef}
\mathcal{T}_\mathbf{r}(\mathbf{z}) = \mathop{\mathrm{Lim}}_{k \to \infty} \mathbf{M}_\mathbf{r}^k\, \tau_\mathbf{r}^{-k}(\mathbf{z}) \qquad (\mathbf{z} \in \mathbb{Z}^n)\,.
\end{equation}
It is conjectured that the collection $\mathcal{C}_\mathbf{r} =\{\mathcal{T}_\mathbf{r}(\mathbf{z})\;:\; \mathbf{z}\in \mathbb{Z}^n \}$ forms a tiling of $\mathbb{R}^n$ for every parameter~$\mathbf{r} \in \mathcal{E}_n$. 
As $\mathcal{C}_\mathbf{r}$ is known to be the collection of Rauzy fractals associated with beta numeration for special choices of~$\mathbf{r}$, this tiling conjecture includes the \emph{Pisot conjecture} for beta numeration asserting that the Rauzy fractals associated with each Pisot unit $\beta$ form a tiling, see e.g.\ \cite{Akiyama:02,Berthe-Siegel:05}.

It turns out that for certain choices of $(\mathbf{A},\mathcal{D})$ the above-mentioned intersections of rational self-affine tiles are just affine images of SRS tiles. Indeed, using our tiling theorem we are able to provide a dense subset of parameters $\mathbf{r} \in \mathcal{E}_n$ such that $\mathcal{C}_\mathbf{r}$ forms a tiling of~$\mathbb{R}^n$, see Theorem~\ref{srsthm}. This might be of interest also for the Pisot conjecture, as our tiling parameters are arbitrarily close to each parameter corresponding to beta numeration. For the parameters associated to beta numeration so far it is only known that they form a multiple tiling, see \cite{Berthe-Siegel:05,Kalle-Steiner:12}. 
As our results show that there exists no open subset of~$\mathcal{E}_n$ consisting only of ``non-tiling parameters'', they indicate that such parameters can only occur due to ``algebraic reasons''. 

\begin{remark}
In the framework of symmetric beta numeration, parameters giving rise to double tilings exist~(see \cite{Kalle-Steiner:12}).
Similarly to the classical case, symmetric beta numeration is a special instance of so-called symmetric SRS,
a~variant of SRS that has been introduced in \cite{Akiyama-Scheicher:07}.
The methods developed in Section~\ref{sec:relat-betw-tiles} can be carried over to exhibit a dense set of symmetric SRS parameters that give rise to tilings.
This indicates that the non-tiling parameters are exceptional in this case as well.
\end{remark}

\begin{remark}
Lagarias and Wang also considered integer matrices with reducible characteristic polynomial.
In this case, there exist situations where the tiling property fails.
These situations were characterized in \cite{Lagarias-Wang:97}.
By generalizing our methods, it is also possible to set up a tiling theory for matrices $\mathbf{A}\in\mathbb{Q}^{n\times n}$ with reducible characteristic polynomial.
As one has to keep track of the reducible factors and single out nontrivial Jordan blocks when defining the representation space~$\mathbb{K}_\alpha$, the definitions get more complicated than in the irreducible setting.
Since we wish to concentrate on the main ideas of our new theory in the present paper, we have decided to postpone the treatment of the reducible case to a forthcoming paper.
\end{remark}

\section{Basic definitions and main results} \label{sec:basic-defin-main}

In the present section, we give precise definitions of the classes of tiles to which this paper is devoted and state our main results. We start with some preparations and notations.

\subsection*{$\mathfrak{p}$-adic completions}
Let $K$ be a number field. For each given (finite or infinite) prime $\mathfrak{p}$ of~$K$, we choose an absolute value $\lvert\cdot\rvert_\mathfrak{p}$ and write $K_\mathfrak{p}$ for the completion of $K$ with respect to~$\lvert\cdot\rvert_\mathfrak{p}$. In all what follows, the absolute value $\lvert\cdot\rvert_\mathfrak{p}$ is chosen in the following way. Let $\xi \in K$ be given. If $\mathfrak{p} \mid \infty$, denote by $\xi^{(\mathfrak{p})}$ the associated Galois conjugate of~$\xi$. If $\mathfrak{p}$ is real, we set $|\xi|_\mathfrak{p} = |\xi^{(\mathfrak{p})}|$, and if $\mathfrak{p}$ is complex, we set $|\xi|_\mathfrak{p} = |\xi^{(\mathfrak{p})}|^2$.
Finally, if $\mathfrak{p}$ is finite, we put $|\xi|_\mathfrak{p}=\mathfrak{N}(\mathfrak{p})^{-v_\mathfrak{p}(\xi)}$, where $\mathfrak{N}(\cdot)$ is the norm of a (fractional) ideal and $v_\mathfrak{p}(\xi)$ denotes the exponent of $\mathfrak{p}$ in the prime ideal decomposition of the principal ideal~$(\xi)$. Note that in any case $\lvert\cdot\rvert_\mathfrak{p}$ induces a metric on $K_\mathfrak{p}$. If $\mathfrak{p} \mid \infty$, then we equip $K_\mathfrak{p}$ with the real Lebesgue measure in case $K_\mathfrak{p} = \mathbb{R}$ and with the complex Lebesgue measure otherwise.
If $\mathfrak{p} \nmid \infty$, then $\mathfrak{p}$ lies over the rational prime $p$ satisfying $(p) = \mathfrak{p} \cap \mathbb{Z}$.
In this case, we equip $K_\mathfrak{p}$ with the Haar measure $\mu_\mathfrak{p}(a + \mathfrak{p}^k) = \mathfrak{N}(\mathfrak{p})^{-k}= p^{-k f(\mathfrak{p})}$, where $f(\mathfrak{p})$ denotes the inertia degree of $\mathfrak{p}$ over~$p$.
For details, we refer to \cite[Chapter~I, \S8, and Chapter~III, \S1]{Neukirch:99}.

\subsection*{Representation space~$\mathbb{K}_\alpha$}
Throughout the paper, let $\alpha$ be an expanding algebraic number with primitive minimal polynomial 
\begin{equation}\label{minpol}
A(X) = a_n X^n + a_{n-1} X^{n-1} + \cdots + a_1 X + a_0 \in \mathbb{Z}[X]\,.
\end{equation}
Here, \emph{expanding} means that every root of $A$ is outside the unit circle (which implies that $|a_0| \ge 2$), and \emph{primitive} means that $(a_0, a_1, \ldots, a_n) = 1$.
A~sufficient condition for~$A$ to be expanding is given by $|a_0| > a_1+\cdots + a_n$; moreover, the Schur-Cohn Algorithm can be used to check whether a given polynomial is expanding or not; see e.g.\ \cite[pp.~491--494]{Henrici:88}.

The ring of integers of the number field $K = \mathbb{Q}(\alpha)$ will be denoted by~$\mathcal{O}$.
Let
\[
\alpha\, \mathcal{O} = \frac{\mathfrak{a}}{\mathfrak{b}}\,, \qquad (\mathfrak{a}, \mathfrak{b}) = \mathcal{O}\,,
\]
where $\mathfrak{a},\mathfrak{b}$ are ideals in~$\mathcal{O}$, $S_\alpha=\{\mathfrak{p}:\, \mathfrak{p} \mid \infty\ \hbox{or}\ \mathfrak{p} \mid \mathfrak{b}\}$, and define the representation space
\[
\mathbb{K}_\alpha = \prod_{\mathfrak{p} \in S_\alpha} K_\mathfrak{p} = \mathbb{K}_\infty \times \mathbb{K}_\mathfrak{b}\,, \quad \mbox{with} \quad \mathbb{K}_\infty =  \prod_{\mathfrak{p}\mid\infty} K_\mathfrak{p}\quad \mbox{and} \quad\ \mathbb{K}_\mathfrak{b} = \prod_{\mathfrak{p}\mid\mathfrak{b}} K_{\mathfrak{p}}\,.
\]
Moreover, $\mathbb{K}_\infty = \mathbb{R}^r \times \mathbb{C}^s$ when $\alpha$ has $r$ real and $s$ pairs of complex Galois conjugates.
The elements of $\mathbb{Q}(\alpha)$ are naturally represented in $\mathbb{K}_\alpha$ by the canonical ring homomorphism
\[
\Phi_\alpha:\, \mathbb{Q}(\alpha) \to \mathbb{K}_\alpha\,, \quad \xi \mapsto\prod_{\mathfrak{p}\in S_\alpha} \xi\,.
\]
We equip $\mathbb{K}_\alpha$ with the product metric of the metrics $\lvert\cdot\rvert_\mathfrak{p}$ and the product measure $\mu_\alpha$ of the measures~$\mu_\mathfrak{p}$, $\mathfrak{p}\in S_\alpha$.
Note that $\mathbb{Q}(\alpha)$ acts multiplicatively on $\mathbb{K}_\alpha$ by
\[
\xi \cdot (z_\mathfrak{p})_{\mathfrak{p}\in S_\alpha} = (\xi z_\mathfrak{p})_{\mathfrak{p}\in S_\alpha} \qquad (\xi \in \mathbb{Q}(\alpha)).
\]
We also use the canonical ring homomorphisms
\[
\Phi_\infty:\, \mathbb{Q}(\alpha) \to \mathbb{K}_\infty\,, \quad \xi \mapsto \prod_{\mathfrak{p}\mid\infty} \xi\,, \qquad \mbox{and} \qquad \Phi_\mathfrak{b}:\, \mathbb{Q}(\alpha) \to \mathbb{K}_\mathfrak{b}\,, \quad \xi \mapsto \prod_{\mathfrak{p}\mid\mathfrak{b}} \xi\,.
\]
The canonical projections from $\mathbb{K}_\alpha$ to~$\mathbb{K}_\infty$ and $\mathbb{K}_\mathfrak{b}$ will be denoted by $\pi_\infty$ and~$\pi_\mathfrak{b}$, respectively, and $\mu_\infty$ denotes the Lebesgue measure on~$\mathbb{K}_\infty$.

\subsection*{Rational self-affine tiles: definition and tiling theorem}
As mentioned in the introduction, we define rational self-affine tiles first in terms of algebraic numbers.

\begin{definition}
Let $\alpha$ be an expanding algebraic number and $\mathcal{D} \subset \mathbb{Z}[\alpha]$. 
The non-empty compact set $\mathcal{F}=\mathcal{F}(\alpha,\mathcal{D}) \subset \mathbb{K}_\alpha$ defined by the set equation
\begin{equation}\label{seteq}
\alpha \cdot \mathcal{F} = \bigcup_{d\in \mathcal{D}} \big(\mathcal{F} + \Phi_\alpha(d)\big)
\end{equation}
is called a \emph{rational self-affine tile} if $\mu_\alpha(\mathcal{F}) > 0$.
\end{definition}

It is immediate from this definition that
\[
\mathcal{F} = \bigg\{\sum_{k=1}^\infty \Phi_\alpha(d_k \alpha^{-k}):\, d_k \in \mathcal{D}\bigg\}\,.
\]
Moreover, the set $\mathcal{F}$ does not depend on the choice of the root~$\alpha$ of the polynomial~$A$.
In analogy to digit sets of integral self-affine tiles, we call $\mathcal{D}$ a
\begin{itemize}
\itemsep3pt
\item \emph{standard digit set} for $\alpha$ if $\mathcal{D}$ is a complete set of coset representatives of the group $\mathbb{Z}[\alpha] / \alpha \mathbb{Z}[\alpha]$ (which implies that $\# \mathcal{D} = |a_0|$), 
\item
\emph{primitive digit set} for $\alpha$ if $\mathbb{Z}\langle \alpha, \mathcal{D}\rangle = \mathbb{Z}[\alpha]$, where $\mathbb{Z}\langle \alpha, \mathcal{D}\rangle$ is the smallest $\alpha$-invariant $\mathbb{Z}$-submodule of $\mathbb{Z}[\alpha]$ containing the difference set $\mathcal{D} - \mathcal{D}$.
\end{itemize}
Note that
\[
\mathbb{Z}\langle \alpha, \mathcal{D}\rangle = \big\langle \mathcal{D} - \mathcal{D}, \alpha (\mathcal{D} - \mathcal{D}), \alpha^2 (\mathcal{D} - \mathcal{D}), \ldots\big\rangle_{\mathbb{Z}}\,.
\]

Rational self-affine tiles can also be defined in terms of rational matrices.
Indeed, let $\mathbf{A} \in \mathbb{Q}^{n\times n}$ be an expanding matrix with irreducible characteristic polynomial~$A$, and let $\mathcal{D} \subset \mathbb{Q}^{n}$. 
Let $\alpha$ be a root of~$A$, and choose a basis of the vector space $\mathbb{Q}(\alpha):\mathbb{Q}$ in a way that the multiplication by $\alpha$ can be viewed as multiplication by $\mathbf{A}$ in this vector space.
Then the non-empty compact set $\mathcal{F} \subset \mathbb{K}_\alpha \simeq \mathbb{R}^n \times \mathbb{K}_\mathfrak{b}$ defined by the set equation
\[
\mathbf{A}\, \mathcal{F} = \bigcup_{d\in \mathcal{D}} \big(\mathcal{F} + \Phi_\alpha(d)\big)
\]
is exactly the same as the set defined in~\eqref{seteq}.

Our first main result contains fundamental properties of rational self-affine tiles.
Before stating it, we define the notion of (multiple) tiling in our context. 

\begin{definition}
Let $(X,\Sigma,\mu)$ be a measure space. 
A~collection $\mathcal{C}$ of compact subsets of~$X$ is called a \emph{multiple tiling} of~$X$ if there exists a positive integer~$m$ such that $\mu$-almost every point of~$X$ is contained in exactly $m$ elements of~$\mathcal{C}$. 
If $m=1$, then $\mathcal{C}$ is called a \emph{tiling} of~$X$.
\end{definition}

\begin{theorem}\label{basicProperties}
Let $\alpha$ be an expanding algebraic number and let $\mathcal{D}$ be a standard digit set for~$\alpha$.
Then the following properties hold for the rational self-affine tile $\mathcal{F} = \mathcal{F}(\alpha,\mathcal{D})$.
\renewcommand{\theenumi}{\roman{enumi}}
\begin{enumerate}
\item \label{property1}
$\mathcal{F}$ is a compact subset of $\mathbb{K}_\alpha$.
\item \label{property2}
$\mathcal{F}$ is the closure of its interior.
\item \label{property3}
The boundary $\partial\mathcal{F}$ of $\mathcal{F}$ has measure zero, i.e., $\mu_\alpha(\partial\mathcal{F})=0$.
\item \label{property5}
The collection $\{\mathcal{F} + \Phi_\alpha(x):\, x \in \mathbb{Z}[\alpha]\}$ forms a multiple tiling of~$\mathbb{K}_\alpha$.
\end{enumerate}
\end{theorem}

The proof of this result is given in Section~\ref{sec:prop-rati-self}.
There, we also show that $\{\mathcal{F} + \Phi_\alpha(x):\, x \in \mathbb{Z}\langle \alpha, \mathcal{D}\rangle\}$ forms a multiple tiling of~$\mathbb{K}_\alpha$.
With considerably more effort, we are able to sharpen this result.
Indeed, our second main result is the following general tiling theorem for rational self-affine tiles, which is proved in Sections~\ref{sec:tiling-criteria} and~\ref{sec:tilingcriterion}.

\begin{theorem} \label{newtilingtheorem}
Let $\alpha$ be an expanding algebraic number and let $\mathcal{D}$ be a standard digit set for~$\alpha$.
Then $\{\mathcal{F} + \Phi_\alpha(x):\, x \in \mathbb{Z}\langle \alpha, \mathcal{D}\rangle\}$ forms a tiling of~$\mathbb{K}_\alpha$.
\end{theorem}

For primitive digit sets, we get the following immediate corollary.
Note that in particular $\{0,1\} \subset \mathcal{D}$ implies primitivity of the digit set.

\begin{corollary} \label{c:1}
Let $\alpha$ be an expanding algebraic number and let $\mathcal{D}$ be a primitive, standard digit set for~$\alpha$.
Then $\{\mathcal{F} + \Phi_\alpha(x):\, x \in \mathbb{Z}[\alpha]\}$ forms a tiling of~$\mathbb{K}_\alpha$.
\end{corollary}

\subsection*{Tiles in $\mathbb{R}^n$ and shift radix systems}
A~second objective of the present paper is the investigation of tiles that are subsets of $\mathbb{K}_\infty \simeq \mathbb{R}^n$.

\begin{definition}
For a given rational self-affine tile $\mathcal{F}=\mathcal{F}(\alpha,\mathcal{D})$, we define the sets
\[
\mathcal{G}(x) = \big\{(z_\mathfrak{p})_{\mathfrak{p}\in S_\alpha} \in \mathcal{F} + \Phi_\alpha(x):\, z_\mathfrak{p}=0\ \hbox{for each}\ \mathfrak{p}\mid \mathfrak{b}\big\} \qquad \big(x \in \mathbb{Z}[\alpha]\big)\,.
\]
The set $\mathcal{G}(x)$ is the intersection of $\mathcal{F} + \Phi_\alpha(x)$ with $\mathbb{K}_\infty \times \Phi_\mathfrak{b}(\{0\})$.
For this reason, we call $\mathcal{G}(x)$ the \emph{intersective tile} at~$x$.
\end{definition}

We will show in Proposition~\ref{p:slices} that $\mathcal{F}$ is essentially made up of slices of translated copies of $\mathcal{G}(x)$, $x \in \mathbb{Z}\langle \alpha, \mathcal{D}\rangle$, more precisely, we will prove that
\begin{equation} \label{e:slices}
\mathcal{F} = \overline{\bigcup_{x \in \mathbb{Z}\langle \alpha, \mathcal{D}\rangle} \big(\mathcal{G}(x) - \Phi_\alpha(x)\big)}\,.
\end{equation}

We will often identify $\mathbb{K}_\infty \times \Phi_\mathfrak{b}(\{0\})$ with~$\mathbb{K}_\infty$, in particular we will then regard $\mathcal{G}(x)$ as a subset of $\mathbb{K}_\infty \simeq \mathbb{R}^n$.
Our third main result is a tiling theorem for intersective tiles.

\begin{theorem} \label{srstiling}
Let $\alpha$ be an expanding algebraic number and let $\mathcal{D}$ be a standard digit set for~$\alpha$. Then the following assertions hold.
\renewcommand{\theenumi}{\roman{enumi}}
\begin{enumerate}
\itemsep1ex
\item \label{three:i}
$\mu_\infty(\partial \mathcal{G}(x))=0$ for each $x\in\mathbb{Z}[\alpha]$.
\item \label{three:ii}
The collection $\{\mathcal{G}(x):\, x \in \mathbb{Z}\langle \alpha, \mathcal{D}\rangle\}$ forms a tiling of $\mathbb{K}_\infty \simeq \mathbb{R}^n$.
\end{enumerate}
\end{theorem}

Recall that the translation set $\mathbb{Z}\langle \alpha, \mathcal{D}\rangle$ is equal to~$\mathbb{Z}[\alpha]$ in the case of a primitive digit set~$\mathcal{D}$.

In our tiling definition, we have not excluded that some tiles are empty.
Indeed, many intersective tiles $\mathcal{G}(x)$, $x \in \mathbb{Z}[\alpha]$, are empty when $A$ is not monic.
To be more precise, set
\[
\Lambda_{\alpha,m} = \mathbb{Z}[\alpha] \cap \alpha^{m-1} \mathbb{Z}[\alpha^{-1}] \qquad (m \in \mathbb{Z}).
\]
We will see in Lemma~\ref{l:G} that $\mathcal{G}(x)$ can be represented in terms of this $\mathbb{Z}$-module, with $m$ chosen in a way that $\mathcal{D} \subset \alpha^{m}\mathbb{Z}[\alpha^{-1}]$.
As an immediate consequence of this representation, we get that $\mathcal{G}(x) = \emptyset$ for all $x \in \mathbb{Z}[\alpha] \setminus \Lambda_{\alpha,m}$.
If $\mathcal{D}$ contains a complete residue system of $\alpha^m \mathbb{Z}[\alpha^{-1}] / \alpha^{m-1} \mathbb{Z}[\alpha^{-1}]$, then Lemma~\ref{l:continuation} shows that these are the only $x \in \mathbb{Z}[\alpha]$ with $\mathcal{G}(x) = \emptyset$.
Moreover, if $x-y$ is in the sublattice $\Lambda_{\alpha,m-k}$ of $\Lambda_{\alpha,m}$ for some large integer~$k$, then the tiles $\mathcal{G}(x) - \Phi_\infty(x)$ and $\mathcal{G}(y) - \Phi_\infty(y)$ are close to each other in Hausdorff metric, i.e., the tiling formed by the collection $\{\mathcal{G}(x):\, x \in \Lambda_{\alpha,m} \cap \mathbb{Z}\langle \alpha, \mathcal{D}\rangle\}$ is almost periodic.
Summing up, we get the following theorem.

\begin{theorem} \label{t:24}
Let $\alpha$ be an expanding algebraic number, let $\mathcal{D}$ be a standard digit set for~$\alpha$, and choose $m \in \mathbb{Z}$ such that $\mathcal{D} \subset \alpha^{m}\mathbb{Z}[\alpha^{-1}]$.
Then $\{\mathcal{G}(x):\, x \in \Lambda_{\alpha,m} \cap \mathbb{Z}\langle \alpha, \mathcal{D}\rangle\}$ forms a tiling of $\mathbb{K}_\infty \simeq \mathbb{R}^n$.

If moreover $\mathcal{D}$ contains a complete residue system of $\alpha^m \mathbb{Z}[\alpha^{-1}] / \alpha^{m-1} \mathbb{Z}[\alpha^{-1}]$, the following assertions hold.
\renewcommand{\theenumi}{\roman{enumi}}
\begin{enumerate}
\itemsep1ex
\item \label{Gnotempty}
Let $x\in\mathbb{Z}[\alpha]$. Then $\mathcal{G}(x) \ne \emptyset$ if and only if $x \in \Lambda_{\alpha,m}$.
\item \label{almostperiodic}
There exists a constant $c > 0$ such that
\[
\delta_H\big(\mathcal{G}(x) - \Phi_\infty(x), \mathcal{G}(y) - \Phi_\infty(y)\big) \le c \max_{\mathfrak{p}\mid\infty} |\alpha^{-k}|_\mathfrak{p}
\]
for all $x,y \in \Lambda_{\alpha,m}$ with $x-y \in \Lambda_{\alpha,m-k}$, $k \ge 0$, where $\delta_H(Y,Z)$ denotes the Hausdorff distance with respect to some metric on $\mathbb{K}_\infty \simeq \mathbb{R}^n$.
\end{enumerate}
\end{theorem}

Again, the intersection with $\mathbb{Z}\langle \alpha, \mathcal{D}\rangle$ can clearly be omitted if $\mathcal{D}$ is a primitive digit set.

In the special instance where $\mathcal{D} = \{0,1,\ldots,|a_0|-1\}$, the conditions of Theorem~\ref{t:24} are satisfied with $m = 0$.
Furthermore, we get a relation to shift radix systems.
To make this precise, associate the parameter $\mathbf{r} = (\frac{a_n}{a_0},\ldots,\frac{a_1}{a_0})$ with the minimal polynomial $a_n X^n + \cdots + a_1 X + a_0$ of~$\alpha$ and let $\mathbf{M}_\mathbf{r}$ be the companion matrix of~$\mathbf{r}$.
As $\alpha$ is expanding, the vector $\mathbf{r}$~lies in the Schur-Cohn region $\mathcal{E}_n = \{\mathbf{r}\in\mathbb{R}^n:\, \varrho(\mathbf{M}_\mathbf{r}) < 1\}$.
Then the collection $\{\mathcal{G}(x):\, x \in \Lambda_{\alpha,0}\}$ of intersective tiles is --- up to a linear transformation --- equal to the collection $\{\mathcal{T}_\mathbf{r}(\mathbf{z}):\, \mathbf{z}\in \mathbb{Z}^n \}$ of SRS tiles defined in \eqref{srstiledef}, see Proposition~\ref{p:SRS}.
In \cite{BSSST:11}, it could not be shown that these collections of tiles always form tilings.
With the help of Theorem~\ref{t:24}, we are now able to fill this gap.

\begin{theorem} \label{srsthm}
Let $a_n X^n + \cdots + a_1 X + a_0 \in \mathbb{Z}[X]$ be an expanding, irreducible polynomial.
Then the collection of SRS tiles $\{\mathcal{T}_\mathbf{r}(\mathbf{z}):\, \mathbf{z}\in \mathbb{Z}^n\}$ forms a tiling of~$\mathbb{R}^n$ for $\mathbf{r} = (\frac{a_n}{a_0},\ldots,\frac{a_1}{a_0})$.
\end{theorem}

Thus we exhibited a dense subset of parameters $\mathbf{r}\in \mathcal{E}_n$ that give rise to a tiling.

\subsection*{Examples}
We now provide two examples in order to illustrate the main results of the present paper.
The first example deals with a rational base number system.
The arithmetics of such number systems was studied in \cite{Akiyama-Frougny-Sakarovitch:07}, where also interesting relations to Mahler's $\frac{3}{2}$-problem \cite{Mahler:68} (which is already addressed by Vijayaraghavan~\cite{Vijayaraghavan:40} and remains still unsolved) were exhibited.

\begin{example}\label{ex:1}
Let $\alpha = \frac{3}{2}$ and $\mathcal{D} = \{0,1,2\}$.
In this example, we have $K = \mathbb{Q}$, $\mathcal{O} = \mathbb{Z}$, thus $\alpha \mathcal{O} = \frac{(3)}{(2)}$, which leads to the representation space $\mathbb{K}_{3/2} = \mathbb{R} \times \mathbb{Q}_2$.
According to Theorem~\ref{basicProperties}, the rational self-affine tile $\mathcal{F} = \mathcal{F}(\frac{3}{2}, \{0,1,2\})$ is a compact subset of~$\mathbb{K}_{3/2}$, which satisfies $\mathcal{F} = \overline{\mathrm{int}(\mathcal{F})}$ and $\mu_{3/2}(\partial \mathcal{F}) = 0$.
Moreover, Theorem~\ref{newtilingtheorem} implies that the collection $\{\mathcal{F} + \Phi_{3/2}(x):\, x \in \mathbb{Z}[\frac{3}{2}]\}$, which is depicted in Figure~\ref{fig:m23}, forms a tiling of~$\mathbb{K}_{3/2}$.

\begin{figure}[ht]
\includegraphics{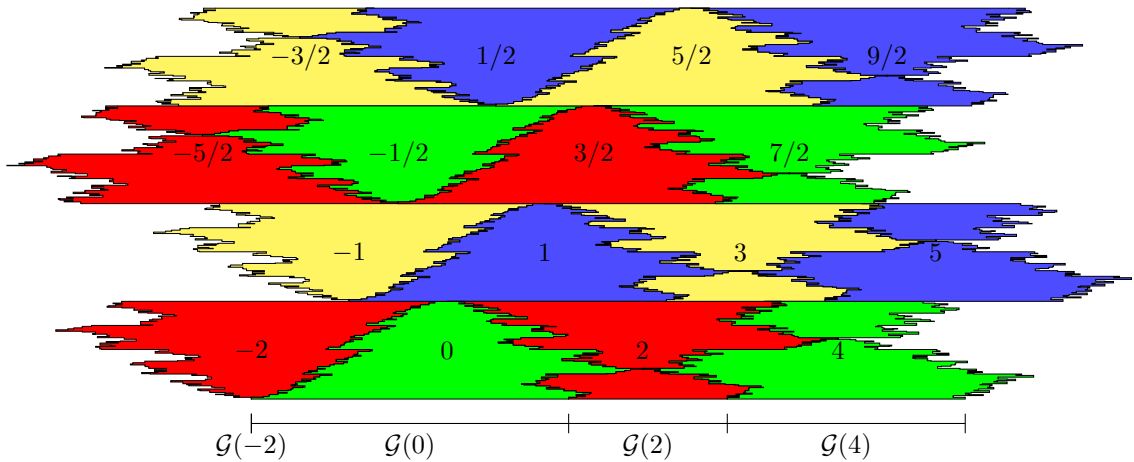}
\caption{The tiles $\mathcal{F} + \Phi_\alpha(x) \in \mathbb{R} \times \mathbb{Q}_2$ for $\alpha=\frac{3}{2}$, $\mathcal{D} = \{0,1,2\}$, $x \in \{\frac{-5}{2},\frac{-4}{2},\ldots,\frac{10}{2}\}$.
Here, an element $\sum_{j=k}^\infty b_j \alpha^{-j}$ of $\mathbb{Q}_2$, with $b_j \in \{0,1\}$, is represented by $\sum_{j=k}^\infty b_j 2^{-j}$.
The intersection of $\mathcal{F} + \Phi_\alpha(x)$ with $\mathbb{R} \times \{0\}$ is equal to~$\mathcal{G}(x)$.} \label{fig:m23}
\end{figure}

The tile labeled by ``$0$'' in this figure is equal to~$\mathcal{F}=\mathcal{F} + \Phi_{3/2}(0)$, and the set equation
\[
\tfrac{3}{2} \cdot \mathcal{F} = \big(\mathcal{F} + \Phi_{3/2}(0)\big) \cup \big(\mathcal{F} + \Phi_{3/2}(1)\big) \cup \big(\mathcal{F} + \Phi_{3/2}(2)\big)
\]
can be seen in the picture.
Due to the embedding of $\mathbb{Q}_2$ into~$\mathbb{R}$, the translated tiles $\mathcal{F} + \Phi_{3/2}(x)$ appear with different shapes in the picture, though they are congruent in~$\mathbb{K}_{3/2}$. 

In this example, the intersective tiles are defined by
\[
\mathcal{G}(x) = \big\{z \in \mathbb{R}:\, (z,0) \in \mathcal{F} + (x,x) \subset \mathbb{K}_{3/2} = \mathbb{R} \times \mathbb{Q}_2\big\}\quad (x\in \mathbb{Z}[\tfrac32])\,;
\]
in particular, $\mathcal{G}(0) = \mathcal{F} \cap (\mathbb{R}\times \{0\})$. Since $\mathbb{Z}\langle \frac32,\{0,1,2\}\rangle = \mathbb{Z}[\frac32]$, Theorem~\ref{srstiling} yields that the intersective tile $\partial \mathcal{G}(x)$ has measure zero for each $x\in\mathbb{Z}[\tfrac32]$ and that the collection $\{\mathcal{G}(x):\, x \in \mathbb{Z}[\tfrac32]\}$ forms a tiling of $\mathbb{K}_\infty=\mathbb{R}$. 

Observe that $\{0,1\} \subset \mathcal{D}$ is a complete residue system of $\mathbb{Z}[\frac{2}{3}]/\frac{2}{3}\mathbb{Z}[\frac{2}{3}] = \mathbb{Z}[\frac{1}{3}]/2\mathbb{Z}[\frac{1}{3}]$.
Therefore, Theorem~\ref{t:24} implies that $\mathcal{G}(x) \ne \emptyset$ if and only if $x \in \Lambda_{3/2,0} = 2 \mathbb{Z}$, and $\{\mathcal{G}(x):\, x \in 2 \mathbb{Z}\}$ forms a tiling of~$\mathbb{R}$.
As shown in \cite[Corollary~5.20]{BSSST:11}, the intersective tiles are (possibly degenerate) intervals with infinitely many different lengths in this case.   
Some of them are depicted in Figure~\ref{fig:m23}.
Note that $\mathcal{G}(-2)$ is equal to the singleton~$\{0\}$ and therefore an example of an intersective tile that is not the closure of its interior (see \cite[Example~3.12]{BSSST:11} for another example of that kind).
Moreover, according to Theorem~\ref{t:24}~(\ref{almostperiodic}), $\mathcal{G}(x)-x$ and $\mathcal{G}(y)-y$ are close to each other (with respect to the Hausdorff distance) if $x-y$ is divisible by a large power of~$2$.

We mention that $\mathcal{G}(0) = [0, 2 K(3)]$, where $K(3) = 1.62227 \cdots$ is related to the solution of the Josephus problem, cf.\ \cite[Corollary~1]{Odlyzko-Wilf:91}. For a discussion of this relation we refer the reader to \cite[Section~4.4 and Theorem~2]{Akiyama-Frougny-Sakarovitch:07}.
\end{example}

\begin{remark}
For the choice $\alpha=\frac32$ and $\mathcal{D}=\{0,2,4\}$ the collection  $\{\mathcal{F} + \Phi_{3/2}(x):\, x \in \mathbb{Z}[\frac{3}{2}]\}$ forms a multiple tiling of $\mathbb{K}_{3/2}$ since $\mathcal{D}$ is not a primitive digit set.
More precisely, almost every point of $\mathbb{K}_{3/2}$ is covered twice by this collection.
\end{remark}

\begin{example}
\begin{figure}[ht]
\includegraphics{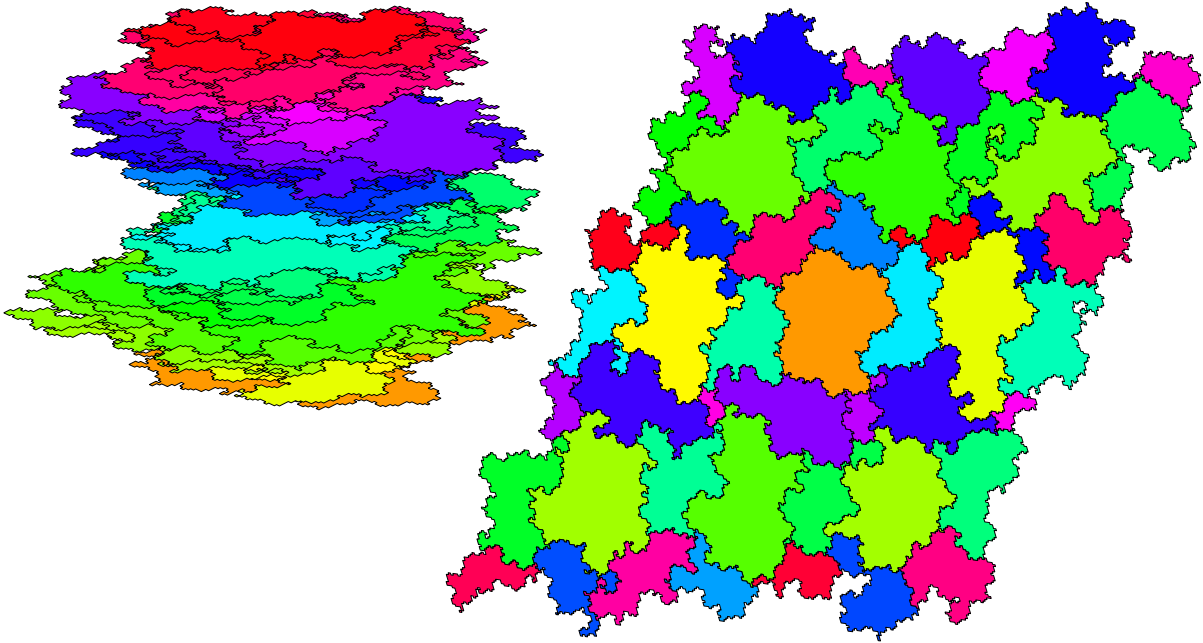}
\caption{The fundamental domain $\mathcal{F} \in \mathbb{C} \times K_{(2,\, 1+\sqrt{-5})}$ for $\alpha = \frac{1+\sqrt{-5}}{2}$ and $\mathcal{D} = \{0,1,2\}$, represented as a ``tower'' of its cuts at several fixed $(2,1+\sqrt{-5})$-adic coordinates (left), and the corresponding intersective tiles $\mathcal{G}(2x_0 + (2\alpha+2) x_1)$, with $x_0, x_1 \in \{-3,-2,\ldots,3\}$ (right).
An element $\sum_{j=k}^\infty b_j \alpha^{-j}$ of $K_{(2,\, -1+\sqrt{-5})}$, with $b_j \in \{0,1\}$, is represented by $\sum_{j=k}^\infty b_j 2^{-j}$.} \label{fig:223}
\end{figure}
Let $\alpha = \frac{-1+\sqrt{-5}}{2}$ be a root of the expanding polynomial $2X^2 + 2X + 3$ and $\mathcal{D} = \{0,1,2\}$.
In this example, we have $K = \mathbb{Q}(\sqrt{-5})$, $\mathcal{O} = \mathbb{Z}[\sqrt{-5}]$, and $\alpha \mathcal{O} = \frac{(3,\,2+\sqrt{-5})}{(2,\,1+\sqrt{-5})}$, hence, the representation space is $\mathbb{K}_\alpha = \mathbb{C} \times K_{(2,\,1+\sqrt{-5})}$.
It is easy to see that $\mathcal{D}$ is a primitive, standard digit set for~$\alpha$.
Therefore, according to Corollary~\ref{c:1}, the collection $\{\mathcal{F} + \Phi_\alpha(x):\, x \in \mathbb{Z}[\alpha]\}$ forms a tiling, whose fundamental domain $\mathcal{F} = \mathcal{F}(\alpha, \mathcal{D})$ is depicted in Figure~\ref{fig:223}.
For the representation of~$\mathcal{F}$, we have used its decomposition into ``slices'' of the form $\mathcal{G}(x) - \Phi_\alpha(x)$, according to~\eqref{e:slices}. Each intersective tile $\mathcal{G}(x)$, $x\in\mathbb{Z}[\alpha]$, satisfies $\mu_\infty(\partial \mathcal{G}(x))=0$ by Theorem~\ref{srstiling}~(\ref{three:i}). Figure~\ref{fig:223} also shows the collection of intersective tiles, which forms again a tiling according to Theorem~\ref{srstiling}~(\ref{three:ii}).
The color of $\mathcal{G}(x)$ on the right hand side of Figure~\ref{fig:223} is the same as the color of $\mathcal{G}(x) - \Phi_\alpha(x)$ in the ``slice representation'' on the left hand side.
By Proposition~\ref{p:SRS}, the collection $\{\mathcal{G}(x):\, x \in \Lambda_{\alpha,0}\}$ is a linear image of the collection $\{\mathcal{T}_\mathbf{r}(\mathbf{z}):\, \mathbf{z} \in \mathbb{Z}^2\}$ of SRS tiles, with $\mathbf{r} = (\frac{2}{3}, \frac{2}{3})$.

Observe that $\{0,1\} \subset \mathcal{D}$ is a complete residue system of $\mathbb{Z}[\alpha^{-1}]/\alpha^{-1}\mathbb{Z}[\alpha^{-1}]$.
Therefore, Theorem~\ref{t:24} implies that $\mathcal{G}(x) \ne \emptyset$ if and only if $x \in \Lambda_{\alpha,0}$, and  $\{\mathcal{G}(x):\, x \in \Lambda_{\alpha,0}\}$ forms a tiling of~$\mathbb{C}$.
According to Lemma~\ref{l:Lambda}, we have $\Lambda_{\alpha,0} = 2 \mathbb{Z} + (2\alpha+2) \mathbb{Z}$.
Here, by Theorem~\ref{t:24}~(\ref{almostperiodic}), $\mathcal{G}(x)-x$ and $\mathcal{G}(y)-y$ are close to each other (with respect to the Hausdorff distance) if $x-y \in \Lambda_{\alpha,-k}$ for some large integer~$k$.
For instance, since $2^k \in \Lambda_{\alpha,-k}$, the tiling $\{\mathcal{G}(x):\, x \in \Lambda_{\alpha,0}\}$ is almost periodic with respect to the lattice $2^k \Lambda_{\alpha,0}$ for large~$k$.
To illustrate this fact, the tiles $\mathcal{G}(0)$ and $\mathcal{G}(2^{k+1})$, $0 \le k \le 8$, are drawn in Figure~\ref{fig:223tiles}.
Note that the shape of $\mathcal{G}(2^9)$ is already very close to that of~$\mathcal{G}(0)$.
\begin{figure}[ht]
\includegraphics{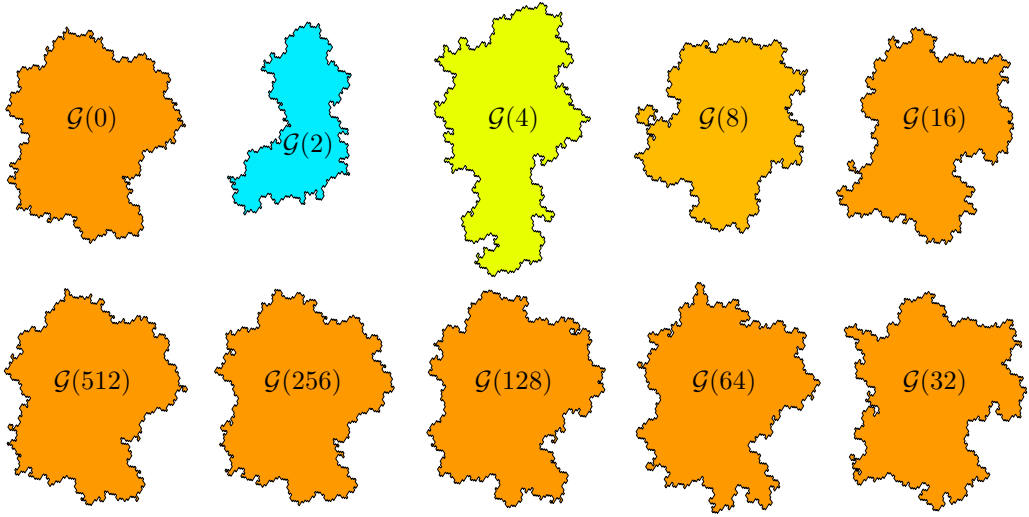}
\caption{The intersective tiles $\mathcal{G}(0)$ and $\mathcal{G}(2^k)$, $1 \le k \le 9$, for $\alpha = \frac{1+\sqrt{-5}}{2}$ and $\mathcal{D} = \{0,1,2\}$.} \label{fig:223tiles}
\end{figure}
\end{example}

\section{Properties of rational self-affine tiles} \label{sec:prop-rati-self}

This section is devoted to fundamental properties of the rational self-affine tile $\mathcal{F} = \mathcal{F}(\alpha, \mathcal{D})$, with $\alpha$ and $\mathcal{D}$ as in Theorem~\ref{basicProperties}.
It is subdivided into two parts.
In the first part we supply auxiliary results that will be needed throughout the paper. The second part is devoted to the proof of Theorem~\ref{basicProperties}.

Since a translation of $\mathcal{D}$ by $x \in \mathbb{Z}[\alpha]$ results in $\mathcal{F}(\alpha, x+ \mathcal{D}) = \mathcal{F}(\alpha, \mathcal{D}) + \Phi_\alpha(x/(\alpha-1))$ and $\mathbb{Z}\langle \alpha, x + \mathcal{D}\rangle = \mathbb{Z}\langle \alpha, \mathcal{D}\rangle$, we can assume w.l.o.g. that $0 \in \mathcal{D}$.
Note that this implies
\[
\mathbb{Z}\langle \alpha, \mathcal{D}\rangle = \langle \mathcal{D}, \alpha \mathcal{D}, \alpha^2 \mathcal{D}, \ldots\rangle_{\mathbb{Z}}\,.
\]
For convenience, in all what follows we use the abbreviation
\[
\mathfrak{z} = \mathbb{Z}\langle \alpha, \mathcal{D}\rangle\,.
\]

\subsection*{Preliminaries}
We start with a basic result on the sets $\mathcal{O}_{S_\alpha}$, $\mathbb{Z}[\alpha]$, and~$\mathfrak{z}$, where
\[
\mathcal{O}_{S_\alpha} = \{x \in \mathbb{Q}(\alpha):\, |x|_{\mathfrak{p}} \le 1\ \mbox{for all}\ \mathfrak{p} \not\in S_\alpha\}
\]
denotes the set of $S_\alpha$-integers.
Recall that a set $M \subset \mathbb{K}_\alpha$ is called a \emph{Delone set} if it is uniformly discrete and relatively dense; i.e., if there are numbers $R>r>0$, such that each ball of radius $r$ contains at most one point of~$M$, and every ball of radius $R$ contains at least one point of~$M$.

\begin{lemma}\label{delonelemma}
The set $\Phi_\alpha(\mathcal{O}_{S_\alpha})$ is a Delone set in $\mathbb{K}_\alpha$.
\end{lemma}

\begin{proof}
The subring $\mathbb{A}_{\mathbb{Q}(\alpha),S_\alpha}$ of $S_\alpha$-ad\`eles in $\mathbb{A}_{\mathbb{Q}(\alpha)}$ (i.e., the ad\`eles which are integral outside~$S_\alpha$) intersects the uniformly discrete subring $\mathbb{Q}(\alpha)$ in~$\mathcal{O}_{S_\alpha}$, so $\Phi_\alpha(\mathcal{O}_{S_\alpha})$ is likewise uniformly discrete in the closed subring~$\mathbb{K}_\alpha$.
In order to show the relative denseness, note that $\mathbb{A}_{\mathbb{Q}(\alpha),S_\alpha}$ is open in~$\mathbb{A}_{\mathbb{Q}(\alpha)}$, so $\mathbb{A}_{\mathbb{Q}(\alpha),S_\alpha}/\mathcal{O}_{S_\alpha}$ (with its quotient topology) is an open subgroup of the compact group $\mathbb{A}_{\mathbb{Q}(\alpha)}/\mathbb{Q}(\alpha)$ and, hence, is compact.
As $\mathbb{K}_\alpha/\Phi_\alpha(\mathcal{O}_{S_\alpha})$ is a quotient of $\mathbb{A}_{\mathbb{Q}(\alpha),S_\alpha}/\mathcal{O}_{S_\alpha}$, it is also compact.
\end{proof}

\begin{lemma} \label{l:subgroup}
The following assertions hold.
\renewcommand{\theenumi}{\roman{enumi}}
\begin{enumerate}
\item \label{OSalpha}
$\mathcal{O}_{S_\alpha} = \mathcal{O}[\alpha]$.
\item \label{Zalphasubroup}
$\mathbb{Z}[\alpha]$ is a subgroup of finite index of~$\mathcal{O}_{S_\alpha}$.
\item \label{zsubgroup}
$\mathfrak{z}$ is a subgroup of finite index of $\mathbb{Z}[\alpha]$.
\end{enumerate}
\end{lemma}

\begin{proof}
We clearly have $\mathcal{O}[\alpha] \subseteq \mathcal{O}_{S_\alpha}$.
For any $x \in \mathcal{O}_{S_\alpha}$, there exists $k \in \mathbb{N}$ such that $x \in \mathfrak{b}^{-k}$.
Since $(\mathfrak{a}, \mathfrak{b}) = \mathcal{O}$, this implies that $x \in (\frac{\mathfrak{a}^k}{\mathfrak{b}^k}, \mathcal{O}) = (\alpha^k \mathcal{O}, \mathcal{O}) \subseteq \mathcal{O}[\alpha]$, which proves~(\ref{OSalpha}).

Recall that $a_n$ is the leading coefficient of the minimal polynomial~$A$ defined in~\eqref{minpol}. As $a_n \alpha \in \mathcal{O}$, the set $\mathbb{Z}[a_n\alpha] \subset \mathbb{Z}[\alpha]$ is an order of $\mathbb{Q}(\alpha)$.
Therefore, there exists $q \in \mathbb{N}$ such that $\mathbb{Z}[a_n \alpha] \subseteq \mathcal{O} \subseteq \frac1q \mathbb{Z}[a_n \alpha]$, thus $\mathbb{Z}[\alpha] \subseteq \mathcal{O}[\alpha] \subseteq \frac1q \mathbb{Z}[\alpha]$.
Suppose that $\mathcal{O}[\alpha] / \mathbb{Z}[\alpha]$ is infinite.
Then there exist $q^n+1$ elements $x_1, x_2, \ldots, x_{q^n+1}$ of $\mathcal{O}[\alpha]$ lying in pairwise distinct congruence classes $\bmod\, \mathbb{Z}[\alpha]$.
Since $x_1, x_2, \ldots, x_{q^n+1} \in \frac{1}{q} \mathbb{Z}[\alpha]$, there exists~$m$ such that $x_1, x_2, \ldots, x_{q^n+1} \in \frac{1}{q} \langle 1, \alpha, \ldots, \alpha^m \rangle_\mathbb{Z}$.
As $\langle 1, \alpha, \ldots, \alpha^m \rangle_\mathbb{Z}$ is a $\mathbb{Z}$-module of rank at most~$n$, the index of $\frac{1}{q} \langle 1, \alpha, \ldots, \alpha^m \rangle_\mathbb{Z} / \langle 1, \alpha, \ldots, \alpha^m \rangle_\mathbb{Z}$ is at most~$q^n$, which implies that $x_i \equiv x_j \bmod\, \langle 1, \alpha, \ldots, \alpha^m \rangle_\mathbb{Z}$ for some $1 \le i < j \le q^n+1$, contradicting that $x_i \not\equiv x_j \bmod\, \mathbb{Z}[\alpha]$.
Therefore, $\mathbb{Z}[\alpha]$ is a subgroup of~$\mathcal{O}[\alpha]$ of index at most~$q^n$.
Using~(\ref{OSalpha}), this gives~(\ref{Zalphasubroup}).

Finally, for each $d \in \mathcal{D} \setminus \{0\}$, we have $\mathfrak{z} \subseteq \mathbb{Z}[\alpha] \subseteq \frac{1}{d}\, \mathfrak{z}$, which implies~(\ref{zsubgroup}).
\end{proof}

In particular, Lemmas~\ref{delonelemma} and~\ref{l:subgroup} yield that $\Phi_\alpha(\mathfrak{z})$ and $\Phi_\alpha(\mathbb{Z}[\alpha])$ are Delone sets in $\mathbb{K}_\alpha$.

Next we show the effect of the action of~$\alpha$ on the measure of a measurable subset of~$\mathbb{K}_\alpha$. To this matter we use the Dedekind-Mertens Lemma for the content of polynomials.
Recall that the \emph{content} $c(f)$ of a polynomial~$f$ over the number field~$K$ is the ideal generated by the coefficients of~$f$, and the Dedekind-Mertens Lemma (see e.g.\ \cite[Section~8]{Anderson:00} or \cite[\S9]{Prufer:32}) asserts that
\begin{equation}\label{DMlemma}
c(f g) = c(f)\, c(g) \qquad (f, g \in K[X])\,.
\end{equation}

\begin{lemma}\label{measuremult}
Let $M \subset \mathbb{K}_\alpha$  be a measurable set. Then
\[
\mu_\alpha(\alpha \cdot M) = \mu_\alpha(M) \prod_{\mathfrak{p}\in S_\alpha} |\alpha|_\mathfrak{p} = |a_0|\, \mu_\alpha(M)\,.
\]
\end{lemma}

\begin{proof}
By the definition of the measure $\mu_\alpha$ and the absolute values $\lvert\cdot\rvert_\mathfrak{p}$, one  immediately gets the first equality. In order to prove the second one, first note that $\prod_{\mathfrak{p}\mid\infty} |\alpha|_\mathfrak{p} = |N_{\mathbb{Q}(\alpha):\mathbb{Q}}(\alpha)| = \frac{|a_0|}{|a_n|}$. Moreover, from the definition of $\lvert\cdot\rvert_\mathfrak{p}$, we get that $\prod_{\mathfrak{p}\mid\mathfrak{b}} |\alpha|_\mathfrak{p} = \mathfrak{N}(\mathfrak{b})$. Combining these identities, we arrive at
\[
\prod_{\mathfrak{p}\in S_\alpha} |\alpha|_\mathfrak{p} = |a_0| \frac{\mathfrak{N}(\mathfrak{b})}{|a_n|}\,.
\]
Now, the lemma follows from
\[
\frac{1}{|a_n|} = c\left(\frac{A(X)}{a_n}\right) = c\big(N_{\mathbb{Q}(\alpha):\mathbb{Q}}(X-\alpha)\big) = \mathfrak{N}\big(c(X-\alpha)\big) = \frac{1}{\mathfrak{N}(\mathfrak{b})}\,,
\]
where the third equality is a consequence of the Dedekind-Mertens Lemma in~\eqref{DMlemma}, and $(\mathcal{O}, \alpha\mathcal{O}) =\big(\mathcal{O}, \frac{\mathfrak{a}}{\mathfrak{b}}\big) = \frac{1}{\mathfrak{b}}$ is used for the last equality.
\end{proof}

\subsection*{Proof of Theorem~\ref{basicProperties}}

We start with the proof of the first assertion of Theorem~\ref{basicProperties}.

\begin{proof}[Proof of Theorem~\ref{basicProperties}~(\ref{property1})]
As $\mathcal{D}$ is a finite set and $|\alpha^{-1}|_\mathfrak{p} < 1$ for all $\mathfrak{p} \in S_\alpha$, the map $(d_k)_{k\ge1} \mapsto \sum_{k=1}^\infty \Phi_\alpha(d_k \alpha^{-k})$ is a continuous map from the compact set of infinite sequences with elements in~$\mathcal{D}$ to $\mathbb{K}_\alpha$. Here, the topology on $\mathcal{D}^\mathbb{N}$ is the usual one, i.e., two sequences are close to each other if the first index where they disagree is large. Therefore, being the continuous image of a compact set, $\mathcal{F}$~is compact.
\end{proof}

To prove the second assertion of Theorem~\ref{basicProperties}, we need the following lemma.
Here, a~collection is \emph{uniformly locally finite} if each open ball meets at most $m$ members of the collection, where $m \in \mathbb{N}$ is a fixed number.

\begin{lemma}\label{lfcovering}
The collection $\{\mathcal{F} + \Phi_\alpha(x): x \in \mathfrak{z}\}$ is a uniformly locally finite covering~of~$\mathbb{K}_\alpha$.
\end{lemma}

\begin{proof}
Since $\mathcal{D}$ is a standard digit set and $\mathcal{D} \subset \mathfrak{z}$, we have $\mathfrak{z} = \alpha\, \mathfrak{z} + \mathcal{D}$.
With~\eqref{seteq} we get
\[
\alpha \cdot \big(\mathcal{F} + \Phi_\alpha(\mathfrak{z})\big) = \mathcal{F} + \Phi_\alpha(\mathcal{D}) + \Phi_\alpha(\alpha\, \mathfrak{z}) = \mathcal{F} + \Phi_\alpha(\mathfrak{z})\,,
\]
and, hence, $\alpha^{-k} \cdot (\mathcal{F} + \Phi_\alpha(\mathfrak{z})) = \mathcal{F} + \Phi_\alpha(\mathfrak{z})$ for all $k\in \mathbb{N}$.
By Lemmas~\ref{delonelemma} and~\ref{l:subgroup}, $\mathcal{F} + \Phi_\alpha(\mathfrak{z})$ is relatively dense in~$\mathbb{K}_\alpha$.
As the action of $\alpha^{-1}$ on $\mathbb{K}_\alpha$ is a uniform contraction, $\mathcal{F} + \Phi_\alpha(\mathfrak{z})$ is even dense in~$\mathbb{K}_\alpha$.
The lemma follows because $\mathcal{F}$ is compact by Theorem~\ref{basicProperties}~\eqref{property1} and $\Phi_\alpha(\mathfrak{z})$ is uniformly discrete by Lemmas~\ref{delonelemma} and~\ref{l:subgroup}.
\end{proof}

We continue with the proof of the remaining assertions of Theorem~\ref{basicProperties}.

\begin{proof}[Proof of Theorem~\ref{basicProperties}~(\ref{property2})]
By a Baire category argument, the uniform local finiteness of the covering $\{\mathcal{F} + \Phi_\alpha(x) \,:\, x \in \mathbb{Z}[\alpha]\}$ established in Lemma~\ref{lfcovering} implies together with the compactness of $\mathcal{F}$ that $\mathcal{F}$ has nonempty interior.
Multiplying~\eqref{seteq} by $\alpha^{-1}$ and iterating this equation $k$ times yields that
\begin{equation}\label{iteratedset}
\mathcal{F} = \bigcup_{d \in \mathcal{D}_k} \alpha^{-k} \cdot \big(\mathcal{F} + \Phi_\alpha(d)\big) \quad \mbox{with} \quad  \mathcal{D}_k = \bigg\{\sum_{j=0}^{k-1} \alpha^j d_j \,:\, d_0,\ldots, d_{k-1} \in \mathcal{D}\bigg\}.
\end{equation}
As the operator $\alpha^{-1}$ acts as a uniform contraction on~$\mathbb{K}_\alpha$, the diameter of $\alpha^{-k} \cdot \mathcal{F}$ tends to zero for $k \to \infty$. Since $k$ can be chosen arbitrarily large and $\alpha^{-k} \cdot (\mathcal{F} + \Phi_\alpha(d))$ contains inner points of~$\mathcal{F}$ for each $d\in \mathcal{D}_k$, the iterated set equation in \eqref{iteratedset} implies that $\mathcal{F}=\overline{\mathrm{int}(\mathcal{F})}$.
\end{proof}

\begin{proof}[Proof of Theorem~\ref{basicProperties}~(\ref{property3})]
Choose $\varepsilon>0$ in a way that $\mathcal{F}$ contains a ball of radius~$\varepsilon$.
Since $\alpha^{-1}$ is a uniform contraction, we may choose $k \in \mathbb{N}$ such that $\mathrm{diam}(\alpha^{-k} \cdot \mathcal{F}) < \frac\varepsilon2$ holds. Thus, for this choice of~$k$, there is at least one $d' \in \mathcal{D}_k$ satisfying $\alpha^{-k} \cdot (\mathcal{F} + \Phi_\alpha(d')) \subset \mathrm{int}(\mathcal{F})$.
Taking measures in the iterated set equation \eqref{iteratedset} now yields
\begin{equation}\label{measureunion}
\mu_\alpha(\mathcal{F}) = \mu_\alpha\bigg(\bigcup_{d \in \mathcal{D}_k} \alpha^{-k} \cdot \big(\mathcal{F} + \Phi_\alpha(d)\big)\bigg)\,.
\end{equation}
As $\alpha^{-k} \cdot (\mathcal{F} + \Phi_\alpha(d')) \subset \mathrm{int}(\mathcal{F})$, the boundary of $\alpha^{-k} \cdot (\mathcal{F} + \Phi_\alpha(d'))$ is covered at least twice by the union in~\eqref{measureunion}.
Thus, using Lemma~\ref{measuremult}, we get from~\eqref{measureunion} that
\[
\mu_\alpha(\mathcal{F}) \le \bigg( \sum_{d \in \mathcal{D}_k} \mu_\alpha(\alpha^{-k} \cdot \mathcal{F}) \bigg) - \mu_\alpha(\alpha^{-k} \cdot \partial\mathcal{F}) = \mu_\alpha(\mathcal{F}) - |a_0|^{-k} \mu_\alpha(\partial \mathcal{F})\,,
\]
which implies that $\mu_\alpha(\partial \mathcal{F}) = 0$.
\end{proof}

\begin{proof}[Proof of Theorem~\ref{basicProperties}~(\ref{property5})]
We show that $\{\mathcal{F} + \Phi_\alpha(x):\, x \in \mathfrak{z}\}$ is a multiple tiling, which implies in view of Lemma~\ref{l:subgroup}~(\ref{zsubgroup}) that the same is true for $\{\mathcal{F} + \Phi_\alpha(x):\, x \in \mathbb{Z}[\alpha]\}$.

Let $m_1<m_2$ be positive integers. Let $M_1$ be the set of points covered by exactly $m_1$ sets of the collection $\{\mathcal{F} + \Phi_\alpha(x):\, x \in \mathfrak{z}\}$, and let $M_2$ be the set of all points covered by $m_2$ or more sets of $\{\mathcal{F} + \Phi_\alpha(x):\, x \in \mathfrak{z}\}$.
Obviously, we have $M_1 \cap M_2 = \emptyset$.
If the assertion is false, then we may choose $m_1$ and $m_2$ with $m_1 < m_2$ in a way that $M_1$ and $M_2$ have positive measure.
Since $\Phi_\alpha(\mathfrak{z})$ is a relatively dense additive group, $M_1$~is a relatively dense subset of~$\mathbb{K}_\alpha$.

Thus it suffices to prove that $M_2$ contains an arbitrarily large ball.
As $\mu_\alpha(\partial\mathcal{F}) = 0$, we can choose $\mathbf{z} \in M_2$ with $\mathbf{z} \not\in \bigcup_{x \in \mathfrak{z}} \partial\big(\mathcal{F} + \Phi_\alpha(x)\big)$.
Therefore, there exists a ball $B_\varepsilon(\mathbf{z})$ centered in $\mathbf{z}$ of radius $\varepsilon>0$ with $B_\varepsilon(\mathbf{z}) \subset M_2$.
As the action of $\alpha$ on $\mathbb{K}_\alpha$ is expanding, it suffices to show that
\begin{equation}\label{inductionM2}
\alpha^k \cdot B_\varepsilon(\mathbf{z}) \subset M_2
\end{equation}
holds for each positive integer~$k$. We do this by induction. The assertion is true for $k=0$. Assume that we already proved it for a certain $k$ and let $\mathbf{z}' \in \alpha^{k} \cdot B_\varepsilon(\mathbf{z})$. Then there exist $x_1, \ldots, x_{m_2} \in \mathfrak{z}$ such that $\mathbf{z}' \in \mathcal{F} + \Phi_\alpha(x_\ell)$ for $\ell \in \{1,\ldots ,m_2\}$, thus $\alpha \cdot \mathbf{z}' \in \alpha \cdot (\mathcal{F} + \Phi_\alpha(x_\ell))$. By the set equation~\eqref{seteq}, this implies that there exist $d_1,\ldots,d_{m_2} \in \mathcal{D}$ such that $\alpha \cdot \mathbf{z}' \in \mathcal{F} + \Phi_\alpha(\alpha x_\ell + d_\ell)$.
Since $\mathcal{D}$ is a standard digit set, we conclude that $\alpha x_\ell  + d_\ell$ are pairwise disjoint elements of~$\mathfrak{z}$, thus $\alpha \cdot \mathbf{z}' \in M_2$. Since $\mathbf{z}'$ was an arbitrary element of $\alpha^k \cdot B_\varepsilon(\mathbf{z})$, this implies that \eqref{inductionM2} is true also for $k+1$ instead of~$k$, and the assertion is established.
\end{proof}

\section{Tiling criteria} \label{sec:tiling-criteria}

In this section, we establish tiling criteria in the spirit of Gr\"ochenig and Haas \cite [Theorem~4.8 and Proposition~5.3]{Groechenig-Haas:94} for the collection $\{\mathcal{F} + \Phi_\alpha(x):\, x \in \mathfrak{z}\}$.

\subsection*{Contact matrix}
The contact matrix governs the neighboring structure of tilings induced by approximations of~$\mathcal{F}$.
In order to define this matrix, note that $\Phi_\alpha(\mathfrak{z})$ is a lattice in~$\mathbb{K}_\alpha$ in the sense that it is a Delone set and an additive group.
Therefore, there exists a compact set $D \subset \mathbb{K}_\alpha$ such that the collection $\{D + \Phi_\alpha(x):\, x \in \mathfrak{z}\}$ forms a tiling of $\mathbb{K}_\alpha$.
Set
\[
\mathcal{V}_0 = \{x \in \mathfrak{z} \setminus \{0\}:\, D \cap (D + \Phi_\alpha(x)) \neq \emptyset\}\,,
\]
define recursively
\[
\mathcal{V}_k = \big\{x \in \mathfrak{z}\setminus\{0\}:\, (\alpha x + \mathcal{D}) \cap \big(y + \mathcal{D}) \neq \emptyset\ \hbox{for some}\ y \in \mathcal{V}_{k-1}\big\}\,,
\]
and let $\mathcal{V} = \bigcup_{k\in \mathbb{N}} \mathcal{V}_k$.
Note that the compactness of $D$ and the fact that multiplication by $\alpha$ is an expanding operator on $\mathbb{K}_\alpha$ imply that $\mathcal{V}$ is a finite set.
The $\#\mathcal{V} \times \#\mathcal{V}$-matrix $\mathbf{C} = (c_{x y})_{x,y \in\mathcal{V}}$ defined by
\[
c_{x y} = \#\big((\alpha x + \mathcal{D}) \cap (y + \mathcal{D})\big) \qquad (x, y \in \mathcal{V})
\]
is called the \emph{contact matrix} of~$\mathcal{F} = \mathcal{F}(\alpha,\mathcal{D})$. (We mention that $\mathbf{C}$ depends on the defining data $\alpha$ and $\mathcal{D}$ of $\mathcal{F}$ as well as on the choice of the compact set $D$.) 

Consider the approximations $\mathcal{F}_k = \bigcup_{d \in \mathcal{D}_k} (D+\Phi_\alpha(d))$ of~$\alpha^k \cdot \mathcal{F}$.
As the sets in the union are measure disjoint, Lemma~\ref{measuremult} yields that
\begin{equation}\label{measureequal}
\mu_\alpha(\alpha^{-k} \cdot \mathcal{F}_k) = \mu_\alpha(D) \qquad (k \in \mathbb{N}).
\end{equation}
The collection $\{\mathcal{F}_k + \Phi_\alpha(\alpha^k x):\, x \in \mathfrak{z}\}$ forms a tiling of~$\mathbb{K}_\alpha$.
By induction on~$k$, we see that
\begin{equation} \label{e:boundaryFk}
\partial \mathcal{F}_k = \bigcup_{x \in \mathcal{V}_k} \mathcal{F}_k \cap (\mathcal{F}_k + \Phi_\alpha(\alpha^k x))\,,
\end{equation}
thus $\alpha^k\, \mathcal{V}_k$ contains all the ``neighbors'' of~$\mathcal{F}_k$ in this tiling.

Our aim in this subsection is to establish a tiling criterion in terms of the spectral radius $\varrho(\mathbf{C})$ of the contact matrix.
The following lemma states that the tiling property of~$\{\mathcal{F} + \Phi_\alpha(x):\, x \in \mathfrak{z}\}$ can be decided by looking at the measure of~$\mathcal{F}$.

\begin{lemma}\label{lem:measuremultiple}
The quantity $\frac{\mu_\alpha(\mathcal{F})}{\mu_\alpha(D)}$ is an integer.
If $\frac{\mu_\alpha(\mathcal{F})}{\mu_\alpha(D)} = 1$, then $\{\mathcal{F} + \Phi_\alpha(x):\, x \in \mathfrak{z}\}$ forms a tiling of~$\mathbb{K}_\alpha$.
\end{lemma}

\begin{proof}
We have seen in Section~\ref{sec:prop-rati-self} that $\{\mathcal{F} + \Phi_\alpha(x):\, x \in \mathfrak{z}\}$ forms a multiple tiling of~$\mathbb{K}_\alpha$.
As $\{D + \Phi_\alpha(x):\, x \in \mathfrak{z}\}$ forms a tiling of~$\mathbb{K}_\alpha$,  the multiplicity is equal to $\frac{\mu_\alpha(\mathcal{F})}{\mu_\alpha(D)}$.
\end{proof}

Let
\[
U_k = \big\{(d,d') \in \mathfrak{z}^2:\, d \in \mathcal{D}_k,\, d' \not\in \mathcal{D}_k\ \hbox{and}\ \big(D+\Phi_\alpha(d)\big) \cap \big(D+\Phi_\alpha(d')\big) \neq \emptyset\big\}\,.
\]
The cardinality of $U_k$ is used in the following criterion.

\begin{proposition}\label{sigmacriterion}
If $\lim_{k\to\infty} |a_0|^{-k}\, \#U_k = 0$, then the collection $\{\mathcal{F} + \Phi_\alpha(x):\, x \in \mathfrak{z}\}$ forms a tiling of~$\mathbb{K}_\alpha$.
\end{proposition}

\begin{proof}
Note that $(d,d') \in U_k$ implies that $d-d' \in \mathcal{V}_0$.
Moreover, we have
\begin{equation}\label{partialSn}
\partial \mathcal{F}_k = \bigcup_{(d,d') \in U_k} \big(D+\Phi_\alpha(d)\big) \cap \big(D+\Phi_\alpha(d')\big)\,.
\end{equation}

For a subset $M \subset \mathbb{K}_\alpha$, let $N(M,\varepsilon)$ be the $\varepsilon$-neighborhood of $M$ with respect to some fixed metric in $\mathbb{K}_\alpha$ and set $R(M,\varepsilon) = N(M,\varepsilon) \setminus M$.
Using this notation, we easily derive by induction on $k$ that (see also \cite[Lemma~4.4]{Groechenig-Haas:94}) there exists some $\varepsilon > 0$ such that
\begin{equation}\label{FinclusionApprox}
\alpha^k \cdot \mathcal{F} \subset N(\mathcal{F}_k, \varepsilon) \quad \hbox{for all}\ k \in \mathbb{N}\,.
\end{equation}
Observe that \eqref{partialSn} implies $\mu_\alpha(R(\mathcal{F}_k, \varepsilon)) \ll \#U_k$ for $k \to \infty$.
Thus, multiplying this by $|a_0|^{-k}$ and using the hypothesis of the proposition yields that
\begin{equation}\label{lim0}
\lim_{k \to \infty} \mu_\alpha\big(\alpha^{-k} \cdot R(\mathcal{F}_k, \varepsilon)\big) = 0\,.
\end{equation}
In view of~\eqref{FinclusionApprox}, we may write
$\mathcal{F} = \big(\mathcal{F} \cap \alpha^{-k} \cdot R(\mathcal{F}_k, \varepsilon)\big) \cup (\mathcal{F} \cap \alpha^{-k} \cdot \mathcal{F}_k)$.
Using \eqref{lim0} and~\eqref{measureequal},  this implies that
\[
\mu_{\alpha}(\mathcal{F}) = \lim_{k\to\infty} \mu_{\alpha}(\mathcal{F} \cap \alpha^{-k} \cdot \mathcal{F}_k) \le \mu_{\alpha}(\alpha^{-k} \cdot \mathcal{F}_k) = \mu_\alpha(D)\,,
\]
and the result follows from Lemma~\ref{lem:measuremultiple}.
\end{proof}

The cardinality of $U_k$ is related to the $k$-th power of the contact matrix~$\mathbf{C}$.

\begin{lemma}\label{sigmacontact}
Let $\mathbf{C}^k = (c_{x y}^{(k)})_{x,y\in\mathcal{V}}$, then $\#U_k = \sum_{x\in\mathcal{V},\,y\in\mathcal{V}_0} c_{x y}^{(k)}$.
\end{lemma}

\begin{proof}
Using~\eqref{e:boundaryFk}, this lemma is proved in the same way as Lemma~4.7 in \cite{Groechenig-Haas:94}.
\end{proof}

We obtain a tiling criterion in terms of the spectral radius $\varrho(\mathbf{C})$.

\begin{proposition} \label{p:48}
If $\varrho(\mathbf{C}) < |a_0|$, then $\{\mathcal{F} + \Phi_\alpha(x):\, x \in \mathfrak{z}\}$ forms a tiling of~$\mathbb{K}_\alpha$.
\end{proposition}

\begin{proof}
As $|a_0|^{-1}\, \mathbf{C}$ is a contraction, we have $\lim_{k\to\infty} |a_0|^{-k}\, c_{x y}^{(k)} = 0$ for all $x, y \in \mathcal{V}$.
Thus, Lemma~\ref{sigmacontact} yields that $|a_0|^{-k}\, \#U_k \to 0$ for $k \to \infty$, and Proposition~\ref{sigmacriterion} gives the result.
\end{proof}

\begin{remark} \label{r:effort}
With some more effort, one can prove that $\varrho(\mathbf{C}) < |a_0|$ if and only if $\{\mathcal{F} + \Phi_\alpha(x):\, x \in \mathfrak{z}\}$ forms a tiling.
However, we will not need this result in the sequel.
\end{remark}

\subsection*{Fourier analytic tiling criterion} 
We now establish a Fourier analytic (non-)tiling criterion.
To this matter, we define the characters (see e.g.\ Section~4 of Tate's thesis \cite{Tate:67})
\begin{align*}
\chi_\mathfrak{p}: &\ K_\mathfrak{p} \to \mathbb{C}\,, \quad z_\mathfrak{p} \mapsto \left\{\begin{array}{ll}\exp(-2\pi i\, \mathrm{Tr}_{K_\mathfrak{p}:\mathbb{R}}(z_\mathfrak{p})) & \mbox{if}\ \mathfrak{p} \mid \infty\,, \\[.5ex] \exp(2 \pi i\, \lambda_p(\mathrm{Tr}_{K_\mathfrak{p}:\mathbb{Q}_p}(z_\mathfrak{p}))) & \mbox{if}\ \mathfrak{p} \mid \mathfrak{b}\ \mbox{and}\ \mathfrak{p} \mid p\,, \end{array}\right. \\[.5ex]
\chi_\alpha: &\ \mathbb{K}_\alpha \to \mathbb{C}\,,  \quad (z_\mathfrak{p})_{\mathfrak{p}\in S_\alpha} \mapsto \prod_{\mathfrak{p} \in S_\alpha} \chi_\mathfrak{p}(z_\mathfrak{p})\,,
\end{align*}
where $\lambda_p(x)$ denotes the fractional part of~$x \in \mathbb{Q}_p$, i.e., $\lambda_p(\sum_{j=k}^\infty b_j p^j) = \sum_{j=k}^{-1} b_j p^j$ for all sequences $(b_j)_{j\ge k}$ with $b_j \in \{0,1,\ldots,p-1\}$, $k < 0$.
In order to set up a suitable Fourier transformation, we need two lemmas on these characters.

\begin{lemma} \label{l:chixi}
For each $\xi \in \mathcal{O}_{S_\alpha}$ we have $\chi_\alpha\big(\Phi_\alpha(\xi)\big) = 1$.
\end{lemma}

\begin{proof}
Let $\mathfrak{p}$ be a prime satisfying $\mathfrak{p} \nmid \mathfrak{b}$ and $\mathfrak{p} \mid p$.
Then $|\xi|_\mathfrak{p} \le 1$ for every $\xi \in \mathcal{O}_{S_\alpha}$, thus we obtain $\mathrm{Tr}_{K_\mathfrak{p}:\mathbb{Q}_p}(\xi) \in \mathbb{Z}_p$, i.e., $\lambda_p(\mathrm{Tr}_{K_\mathfrak{p}:\mathbb{Q}_p}(\xi)) = 0$.
This implies that
\[
\sum_{\mathfrak{p}\mid\mathfrak{b}} \lambda_p(\mathrm{Tr}_{K_\mathfrak{p}:\mathbb{Q}_p}(\xi)) - \mathrm{Tr}_{K:\mathbb{Q}}(\xi) = \sum_{\mathfrak{p}\nmid\infty} \lambda_p(\mathrm{Tr}_{K_\mathfrak{p}:\mathbb{Q}_p}(\xi))  - \mathrm{Tr}_{K:\mathbb{Q}}(\xi) \equiv 0 \pmod{1},
\]
by an argument used in \cite[proof of Lemma~4.1.5]{Tate:67}, and, hence, ${\chi_{\alpha}\big(\Phi_\alpha(\xi)\big) = 1}$.
\end{proof}

Let
\begin{equation}\label{jot}
\mathfrak{z}^* = \big\{\xi \in \mathbb{Q}(\alpha):\, \chi_\alpha\big(\Phi_\alpha(\xi\, x)\big) = 1\ \hbox{for
all}\ x \in \mathfrak{z}\big\}\,,
\end{equation}
and observe that $\mathcal{O}_{S_\alpha} \subseteq \mathfrak{z}^*$ by Lemma~\ref{l:chixi}.
Moreover, denote by $\chi_{\alpha,\xi}$, $\xi \in \mathbb{Q}(\alpha)$, the character defined by $\chi_{\alpha,\xi}(\mathbf{z}) = \chi_\alpha(\xi \cdot \mathbf{z})$.

\begin{lemma} \label{l:ann}
The set $\{\chi_{\alpha,\xi}:\, \xi \in \mathfrak{z}^*\}$ is the annihilator of $\Phi_\alpha(\mathfrak{z})$ in the Pontryagin dual~$\widehat{\mathbb{K}_\alpha}$.
\end{lemma}

\begin{proof}
Set $Y = \{\mathbf{z} \in \mathbb{K}_\alpha:\, \chi_\alpha(x \cdot \mathbf{z}) = 1\ \hbox{for all}\ x \in \mathfrak{z}\}$.
In view of the definition of~$\mathfrak{z}^*$, we only have to show that $Y \subseteq \Phi_\alpha(\mathbb{Q}(\alpha))$.
Since $\mathbb{K}_\alpha / \Phi_\alpha(\mathfrak{z})$ is compact, $Y$~is discrete.
Because $Y$ contains $\Phi_\alpha(\mathcal{O}_{S_\alpha})$, the factor group $Y / \Phi_\alpha(\mathcal{O}_{S_\alpha})$ is a discrete subgroup of the compact group $\mathbb{K}_\alpha / \Phi_\alpha(\mathcal{O}_{S_\alpha})$.
This implies that $Y / \Phi_\alpha(\mathcal{O}_{S_\alpha})$ is finite, thus $Y \subseteq \Phi_\alpha(\mathbb{Q}(\alpha))$.
\end{proof}

We now do Fourier analysis in $D^* = \mathbb{K}_\alpha / \Phi_\alpha(\mathfrak{z}^*)$.
As $\mathcal{O}_{S_\alpha} \subseteq \mathfrak{z}^*$, the factor group $D^*$ is compact, which implies that its Pontryagin dual~$\widehat{D^*}$ is discrete.
By Lemma~\ref{l:ann}, we may write $\widehat{D^*} = \{\chi_{\alpha,x}:\, x \in \mathfrak{z}\}$, cf.\ e.g.\ \cite[Theorem~23.25]{Hewitt-Ross:63}.
Equip $D^*$ with a Haar measure $\mu_{D^*}$ and define the Fourier transform
\[
\widehat{f}(\chi) = F(f)(\chi) = \int_{D^*} f(\mathbf{z})\, \chi(-\mathbf{z})\, d\mu_{D^*}(\mathbf{z}) \qquad(\chi \in \widehat{D^*}),
\]
see \cite[Section~31.46]{Hewitt-Ross:70} for details on how to define the Fourier transform on quotient groups.
Since $\widehat{D^*}$ is discrete, the Fourier inversion formula implies that
\[
f(\mathbf{z}) = F^{-1}(\widehat{f}\,)(\mathbf{z}) = \sum_{\chi \in \widehat{D^*}} \widehat f(\chi)\, \chi(\mathbf{z}) \qquad (\mathbf{z} \in D^*).
\]

Let $\Omega$ denote the set of vectors $(t_x)_{x\in \mathcal{V}}$ with $t_x \in \mathbb{R}$, and denote by $\widehat\Omega$ the set of all functions $f:\, \mathbb{K}_\alpha \to \mathbb{C}$ of the form $f = \sum_{x\in\mathcal{V}} t_x\, \chi_{\alpha,x}$.
We can regard the elements of $\widehat\Omega$ as functions from $D^*$ to~$\mathbb{C}$.
Thus we may apply $F$ to the elements of~$\widehat\Omega$.
Indeed, $F$~and its inverse induce the mappings (called $F$ and $F^{-1}$ again without risk of confusion)
\begin{align*}
F: &\ \widehat\Omega \to \Omega, \quad \sum_{x\in\mathcal{V}} t_x\, \chi_{\alpha,x} \mapsto (t_x)_{x\in \mathcal{V}}\,, \\
F^{-1}: &\ \Omega \to \widehat\Omega,  \quad (t_x)_{x\in\mathcal{V}} \mapsto \sum_{x\in\mathcal{V}} t_x\, \chi_{\alpha,x}\,.
\end{align*}
Following \cite[Section~5]{Groechenig-Haas:94}, we define the Fourier transform $\widehat{\mathbf{C}}$ of the contact matrix $\mathbf{C}$~by
\[
\widehat{\mathbf{C}} = F^{-1} \mathbf{C} F.
\]

We are now in the position to give a version of the Fourier analytic criterion of Gr\"ochenig and Haas~\cite[Proposition~5.3]{Groechenig-Haas:94} that will be used in order to check the tiling property.

\begin{proposition} \label{p:53}
If the collection $\{\mathcal{F} + \Phi_\alpha(x):\, x \in \mathfrak{z}\}$ does not form a tiling, then there exists a non-constant real-valued function $f \in \widehat\Omega$ satisfying $\widehat{\mathbf{C}} f = |a_0|\, f$.
\end{proposition}

\begin{proof}
Suppose that $\{\mathcal{F} + \Phi_\alpha(x):\, x \in \mathfrak{z}\}$ does not form a tiling.
By Proposition~\ref{p:48}, we have $\varrho(\mathbf{C}) \ge |a_0|$.
Since
\[
\sum_{x\in\mathcal{V}} c_{x y} = \sum_{x\in\mathcal{V}} \#\big((\alpha x + \mathcal{D}) \cap (y + \mathcal{D})\big) = \#\bigg(\bigcup_{x\in\mathcal{V}} (\alpha x + \mathcal{D}) \cap (y + \mathcal{D})\bigg) \le |a_0|
\]
for all $y \in\mathcal{V}$, we also have $\varrho(\mathbf{C}) \le |a_0|$.
Therefore, there exists an eigenvector $\mathbf{t} = (t_x)_{x\in\mathcal{V}}$ such that $\mathbf{C}\, \mathbf{t} = |a_0|\, \mathbf{t}$.
Since $\mathcal{V} = -\mathcal{V}$, we can choose~$\mathbf{t}$ in a way that $t_x = t_{-x}$.
Then $F^{-1} \mathbf{t} = \sum_{x\in\mathcal{V}} t_x\, \chi_{\alpha,x}$ is a real-valued eigenfunction of $\widehat{\mathbf{C}}$ to the eigenvalue~$|a_0|$.
Furthermore, $0 \not\in \mathcal{V}$ implies that $F^{-1} \mathbf{t}$ is not constant.
\end{proof}

\begin{remark}
Note that $\widehat{\mathbf{C}} f = |a_0|\, f$ implies that $\mathbf{C} (F f) = |a_0|\, (F f)$, i.e., $\varrho(\mathbf{C}) \ge |a_0|$.
Thus, in view of Remark~\ref{r:effort}, one could in addition show the converse of Proposition~\ref{p:53}.
However, we will show in Section~\ref{sec:tilingcriterion} that non-constant real-valued functions $f \in \widehat\Omega$ satisfying $\widehat{\mathbf{C}} f = |a_0|\, f$ do not exist in our setting.
\end{remark}

\subsection*{Writing $\widehat{\mathbf{C}}$ as a transfer operator}
We conclude this section by deriving a representation of $\widehat{\mathbf{C}}$ as ``transfer operator''.
To this matter, let $\mathcal{D}^*$ be a complete residue system of $\mathfrak{z}^* / \alpha\, \mathfrak{z}^*$.
The set $\mathcal{D}^*$ can be seen as a ``dual'' set of digits.
With help of this set, the character $\chi_\alpha$ can be used to filter the elements of~$\mathfrak{z}$.
This is made precise in the following lemma.

\begin{lemma} \label{l:511}
Let $x \in \mathfrak{z}$.
We have
\[
\frac{1}{|a_0|} \sum_{d^* \in \mathcal{D}^*} \chi_\alpha\big(\Phi_\alpha(\alpha^{-1} x\, d^*)\big) = \left\{\begin{array}{cl}1 & \mbox{if}\ x \in \alpha\, \mathfrak{z}, \\ 0 & \mbox{if}\ x \not\in \alpha\, \mathfrak{z}.\end{array}\right.
\]
\end{lemma}

\begin{proof}
The first alternative in the statement follows from the definition of $\mathfrak{z}^*$ and since $\#\mathcal{D}^* = |a_0|$.
To prove the second one, let $x \in \mathfrak{z}$ and $G(x) = \{\chi_\alpha\big(\Phi_\alpha(\alpha^{-1} x\, \xi)\big):\, \xi \in \mathfrak{z}^*\}$.
Then $G(x)$ is a cyclic group of order dividing~$|a_0|$.
This group is non-trivial for $x \in \mathfrak{z} \setminus \alpha\, \mathfrak{z}$ (see e.g.\ \cite[Corollary~23.26]{Hewitt-Ross:63}), which implies the second alternative.
\end{proof}

Let
\[
u(\mathbf{z}) = \bigg|\frac{1}{|a_0|} \sum_{d \in \mathcal{D}} \chi_\alpha(d \cdot \mathbf{z})\bigg|^2 \qquad (\mathbf{z} \in \mathbb{K}_\alpha)
\]
be the auto-correlation function of the digits $d\in \mathcal{D}$ and set
\[
\tau_{\alpha,d^*}(\mathbf{z}) = \alpha^{-1} \cdot \big(\mathbf{z} + \Phi_\alpha(d^*)\big) \qquad (d^* \in \mathcal{D}^*,\ \mathbf{z} \in \mathbb{K}_\alpha).
\]
We need the following auxiliary result on the function~$u$.

\begin{lemma} \label{l:512}
For each $\mathbf{z} \in \mathbb{K}_\alpha$, we have
\[
\sum_{d^*\in\mathcal{D}^*} u\big(\tau_{\alpha,d^*}(\mathbf{z})\big) = 1\,.
\]
\end{lemma}

\begin{proof}
Similarly to \cite[Lemma~5.1]{Groechenig-Haas:94}, the proof is done by direct calculation.
Indeed, using Lemma~\ref{l:511},
\begin{align*}
\sum_{d^*\in\mathcal{D}^*} u\big(\alpha^{-1} \cdot \big(\mathbf{z} + \Phi_\alpha(d^*)\big)\big) & = \sum_{d^*\in\mathcal{D}^*} \frac{1}{|a_0|^2} \sum_{d,d'\in\mathcal{D}} \chi_\alpha\big(\big((d-d') \alpha^{-1} \cdot \big(\mathbf{z} + \Phi_\alpha(d^*)\big)\big) \\
& \hspace{-2em} = \frac{1}{|a_0|} \sum_{d,d'\in\mathcal{D}} \chi_\alpha\big((d-d') \alpha^{-1} \cdot \mathbf{z}\big) \ \frac{1}{|a_0|} \sum_{d^*\in\mathcal{D}^*} \chi_\alpha\big((d-d') \alpha^{-1} \cdot \Phi_\alpha(d^*)\big) \\
& \hspace{-2em} = \frac{1}{|a_0|} \sum_{d,d'\in\mathcal{D}} \chi_\alpha\big((d-d') \alpha^{-1} \cdot \mathbf{z}\big)\, \delta_{d,d'} = 1\,,
\end{align*}
where $\delta_{d,d'}$ denotes the Kronecker $\delta$-function.
\end{proof}

We conclude this section by establishing the following representation of~$\widehat{\mathbf{C}}$.

\begin{proposition} \label{l:52}
The operator $\widehat{\mathbf{C}}$ can be written as a \emph{transfer operator}:
\begin{equation} \label{e:hatC}
\widehat{\mathbf{C}} f(\mathbf{z}) = |a_0| \sum_{d^*\in\mathcal{D}^*} u\big(\tau_{\alpha,d^*}(\mathbf{z})\big)\, f\big(\tau_{\alpha,d^*}(\mathbf{z})\big) \qquad (f \in \widehat\Omega,\ \mathbf{z} \in \mathbb{K}_\alpha).
\end{equation}
\end{proposition}

\begin{proof}
Let $f = \sum_{y\in\mathcal{V}} t_y\, \chi_{\alpha,y}$.
By definition, we have $F \widehat{\mathbf{C}} f = \mathbf{C} F f$.
Thus it suffices to show that $(F \widehat{\mathbf{C}} f)(x) = (\mathbf{C} F f)(x)$ for all $x \in \widehat{H}$, with $\widehat{\mathbf{C}}$ as in~\eqref{e:hatC}.

We have $(\mathbf{C} F f)(x) = \sum_{y\in\mathcal{V}} c_{x y}\, t_y$ if $x \in \mathcal{V}$, $(\mathbf{C} F f)(x) = 0$ otherwise.
On the other hand, if $\widehat{\mathbf{C}}$ is as in the statement of the lemma, then
\begin{align*}
\widehat{\mathbf{C}} f(\mathbf{z}) & = \frac{1}{|a_0|} \sum_{d^*\in\mathcal{D}^*,\,d,d'\in\mathcal{D}} \chi_\alpha\big((d-d')\, \alpha^{-1} \cdot \big(\mathbf{z} + \Phi_\alpha(d^*)\big)\big) \ \sum_{y\in\mathcal{V}} t_y\, \chi_\alpha\big(y\, \alpha^{-1} \cdot \big(\mathbf{z} + \Phi_\alpha(d^*)\big)\big) \\
& = \sum_{y\in\mathcal{V}} t_y \sum_{d,d'\in\mathcal{D}} \chi_\alpha\big((d-d'+y)\, \alpha^{-1} \cdot \mathbf{z}\big) \ \frac{1}{|a_0|} \sum_{d^*\in\mathcal{D}^*} \chi_\alpha\big((d-d'+y)\, \alpha^{-1} \cdot \Phi_\alpha(d^*)\big)\,.
\end{align*}
By Lemma~\ref{l:511}, the last sum is nonzero if and only if $d-d'+y = \alpha x$ for some $x \in \mathfrak{z}$.
Since $y \in \mathcal{V}$, we also have $x \in \mathcal{V}$ in this case.
Therefore, we obtain
\[
\widehat{\mathbf{C}} f = \sum_{x,y\in\mathcal{V}} t_y\, \#\big((\alpha x + \mathcal{D}) \cap (y + \mathcal{D})\big)\, \chi_{\alpha,x} = \sum_{x,y\in\mathcal{V}} c_{x y}\, t_y\, \chi_{\alpha,x}\,.
\]
This gives $(F \widehat{\mathbf{C}} f)(x) = \sum_{y\in\mathcal{V}} c_{x y}\, t_y$ if $x \in \mathcal{V}$, $(F \widehat{\mathbf{C}} f)(x) = 0$ otherwise.
\end{proof}

\section{The tiling theorem} \label{sec:tilingcriterion}
In the present section, we finish the proof of Theorem~\ref{newtilingtheorem} by studying the eigenfunctions of~$\widehat{\mathbf{C}}$.
Already in \cite{Groechenig-Haas:94}, the extremal values of eigenfunctions of certain transfer operators are studied in order to obtain a tiling theorem.
This approach was considerably generalized in \cite{Lagarias-Wang:97}, where a general theorem on the zero sets of eigenfunctions of transfer operators (established in \cite{Cerveau-Conze-Raugi:96}) was applied to prove the tiling result for integral self-affine tiles mentioned in the introduction.
Here, we further develop this theory.

Similarly to~\cite{Lagarias-Wang:97}, we call $f \in \widehat\Omega$ a \emph{special eigenfunction} if
\[
\widehat{\mathbf{C}} f = |a_0|\, f\,, \quad f(\mathbf{0}) > 0\,, \quad \min_{\mathbf{z} \in \mathbb{K}_\alpha} f(\mathbf{z}) = 0\,.
\]
We get the following lemma.

\begin{lemma} \label{l:special}
If the collection $\{\mathcal{F} + \Phi_\alpha(x):\, x \in \mathfrak{z}\}$ does not form a tiling, then there exists a special eigenfunction $f \in \widehat\Omega$.
\end{lemma}

\begin{proof}
By Proposition~\ref{p:53}, there exists a non-constant real-valued eigenfunction $\tilde{f} \in \widehat\Omega$ with $\widehat{\mathbf{C}} \tilde{f} = |a_0|\, \tilde{f}$.
Then the function $f$ defined by
\[
f(\mathbf{z}) = \left\{\begin{array}{cl}\tilde{f}(\mathbf{z}) - \min_{\mathbf{y}\in\mathbb{K}_\alpha} \tilde{f}(\mathbf{y}) & \mbox{if}\ \tilde{f}(\mathbf{0}) > \min_{\mathbf{y}\in\mathbb{K}_\alpha} \tilde{f}(\mathbf{y}), \\[1ex] \max_{\mathbf{y}\in\mathbb{K}_\alpha} \tilde{f}(\mathbf{y}) - \tilde{f}(\mathbf{z}) & \mbox{otherwise,}\end{array}\right.
\]
is a special eigenfunction.
\end{proof}

Assuming that a special eigenfunction $f \in \widehat\Omega$ exists, we study its (non-empty) zero set
\[
Z_f = \big\{\mathbf{z} \in \mathbb{K}_\alpha:\, f(\mathbf{z}) = 0\big\}\,.
\]
Note first that $Z_f$ is $\Phi_\alpha(\mathfrak{z}^*)$-periodic by the definition of~$\widehat{\Omega}$.
Starting from the assumption that $Z_f$ is non-empty and using self-affinity properties of~$Z_f$, we will obtain the contradictory result $Z_f = \mathbb{K}_\alpha$, which implies that no special eigenfunction exists.
In view of Lemma~\ref{l:special}, this will prove Theorem~\ref{newtilingtheorem}.

\begin{lemma}[{cf.\ \cite[Lemma~3.2]{Lagarias-Wang:97}}] \label{l:32}
Let $f$ be a special eigenfunction and let $\mathcal{D}^*$ be a complete residue system of $\mathfrak{z}^* / \alpha\, \mathfrak{z}^*$.
Then the following assertions hold.
\renewcommand{\theenumi}{\roman{enumi}}
\begin{enumerate}
\itemsep1ex
\item \label{1}
For each $\mathbf{z} \in Z_f$ and $d^* \in \mathcal{D}^*$, $u\big(\tau_{\alpha,d^*}(\mathbf{z})\big) > 0$ implies that $\tau_{\alpha,d^*}(\mathbf{z}) \in Z_f$.
\item \label{2}
For each $\mathbf{z} \in Z_f$, there exists some $d^* \in \mathcal{D}^*$ such that $u(\tau_{\alpha,d^*}(\mathbf{z})) > 0$.
\end{enumerate}
\end{lemma}

\begin{proof}
Let $\mathbf{z} \in Z_f$, then $\widehat{\mathbf{C}} f(\mathbf{z}) = |a_0|\, f(\mathbf{z}) = 0$, thus
\[
\sum_{d^*\in\mathcal{D}^*} u\big(\tau_{\alpha,d^*}(\mathbf{z})\big)\, f\big(\tau_{\alpha,d^*}(\mathbf{z})\big) = 0
\]
by Proposition~\ref{l:52}.
Since $f(\mathbf{z}) \ge 0$ everywhere, every term on the left-hand side must be zero, which shows that $\tau_{\alpha,d^*}(\mathbf{z}) \in Z_f$ if $u\big(\tau_{\alpha,d^*}(\mathbf{z})\big) > 0$.
By Lemma~\ref{l:512}, we have $u\big(\tau_{\alpha,d^*}(\mathbf{z})\big) > 0$ for some $d^* \in \mathcal{D}^*$.
\end{proof}

Lemma~\ref{l:32} motivates the following definitions.
A~set $Y \subset \mathbb{K}_\alpha$ is \emph{$\tau_\alpha$-invariant} (w.r.t.\ a complete residue system $\mathcal{D}^*$ of $\mathfrak{z}^* / \alpha\, \mathfrak{z}^*$) if, for each $\mathbf{z} \in Y$, $d^* \in \mathcal{D}^*$, $u(\tau_{\alpha,d^*}(\mathbf{z})) > 0$ implies $\tau_{\alpha,d^*}(\mathbf{z}) \in Y$.
It is \emph{minimal} $\tau_\alpha$-invariant if it does not contain a proper subset which is also $\tau_\alpha$-invariant.
First we observe that such sets exist (when a special eigenfunction exists).

\begin{lemma} \label{l:NCMI}
Let $f$ be a special eigenfunction.
Then there exists a non-empty compact minimal $\tau_\alpha$-invariant set $Y \subset Z_f$.
\end{lemma}

\begin{proof}
This is proved arguing in the same way as in the proof of Theorem~4.1 of \cite{Lagarias-Wang:97}.
In particular, note that multiplication by $\alpha^{-1}$ is contracting and $u$ is continuous.
\end{proof}

Suppose that the special eigenfunction $f$ is given by $f = \sum_{x\in\mathcal{V}} t_x\, \chi_{\alpha,x}$.
To further explore the zero set~$Z_f$, we apply a result of \cite{Cerveau-Conze-Raugi:96}.
As this result deals only with functions defined on~$\mathbb{R}^n$, we relate $f$ with the function $f_\infty$ defined on~$\mathbb{K}_\infty \simeq \mathbb{R}^n$ by
\[
f_\infty:\ \mathbb{K}_\infty \to \mathbb{R}, \quad (z_\mathfrak{p})_{\mathfrak{p}\mid\infty} \mapsto \sum_{x\in\mathcal{V}} t_x\, \prod_{\mathfrak{p}\mid\infty} \chi_\mathfrak{p}(x z_\mathfrak{p})\,.
\]
By the following lemma, such a relation between $f$ and $f_\infty$ holds for the set
\begin{equation} \label{e:E}
E = \big\{(z_\mathfrak{p})_{\mathfrak{p}\in S_\alpha} \in \mathbb{K}_\alpha:\, |z_\mathfrak{p}|_\mathfrak{p} \le |\alpha^{-m}|_\mathfrak{p}\ \mbox{for all}\ \mathfrak{p} \mid \mathfrak{b}\}\,,
\end{equation}
where $m \in \mathbb{Z}$ is chosen in a way that
\begin{equation}\label{em}
\mathcal{D} \subset \alpha^{m}\mathbb{Z}[\alpha^{-1}]\,.
\end{equation}

\begin{lemma}\label{infinity}
For each $\mathbf{z} \in E$, we have
\[
f(\mathbf{z}) = f_\infty(\pi_\infty(\mathbf{z}))\,.
\]
\end{lemma}

\begin{proof}
For each $x \in \mathcal{V}$, we have $x = \sum_{j=1}^k (d_j - d_j') \alpha^{-j} + \alpha^{-k} x_0$ with $d_j, d_j' \in \mathcal{D} \subset \alpha^{m}\mathbb{Z}[\alpha^{-1}]$, $x_0 \in \mathcal{V}_0$, and arbitrarily large~$k$.
This implies $|x|_\mathfrak{p} \le |\alpha^m|_\mathfrak{p}$ for $\mathfrak{p} \mid \mathfrak{b}$ and, hence, $\chi_\mathfrak{p}(x z_\mathfrak{p}) = 1$ for all $z_\mathfrak{p} \in K_\mathfrak{p}$ with $|z_\mathfrak{p}|_\mathfrak{p} \le |\alpha^{-m}|_\mathfrak{p}$.
This shows that $f(\mathbf{z}) = f_\infty(\pi_\infty(\mathbf{z}))$ for each $\mathbf{z} \in E$.
\end{proof}

We can restrict our attention to~$E$ because of the following lemma.

\begin{lemma} \label{l:YE}
Let $Y$ be a compact  minimal $\tau_\alpha$-invariant set and suppose that $\mathcal{D}^*$ is a complete residue system of $\mathfrak{z}^* / \alpha\, \mathfrak{z}^*$ satisfying
\begin{equation} \label{em2}
|d^*|_\mathfrak{p} \le |\alpha^{-m+1}|_\mathfrak{p} \quad \mbox{for all}\ d^* \in \mathcal{D}^*\mbox{ and all } \mathfrak{p} \mid \mathfrak{b}\,.
\end{equation}
Then $Y$ is contained in~$E$.
\end{lemma}

\begin{proof}
By the minimality of~$Y$, for each $\mathbf{z} \in Y$ there exists $d^* \in \mathcal{D}^*$ such that $\tau_{\alpha,d^*}^{-1}(\mathbf{z}) \in Y$.
Iterating this argument, we obtain $\mathbf{z} = \sum_{j=1}^k \Phi_\alpha(\alpha^{-j} d^*_j) + \alpha^{-k} \cdot \mathbf{z}_0$ with $d^*_j \in \mathcal{D}^*$, $\mathbf{z}_0 \in Y$, and arbitrarily large~$k$.
The result follows now from \eqref{em2} and the compactness of~$Y$.
\end{proof}

We now prove that $\mathcal{D}^*$ can always be chosen to satisfy~\eqref{em2}, by using the well-known \emph{Strong Approximation Theorem} for valuations.

\begin{lemma}[{see e.g.\ \cite[Section~15]{Cassels:67}}]\label{strongapprox}
Let $S$ be a finite set of primes and let $\mathfrak{p}_0$ be a prime of the number field $K$ which does not belong to~$S$.
Let $z_\mathfrak{p} \in K$ be given numbers, for $\mathfrak{p} \in S$.
Then, for every $\varepsilon > 0$, there exists $x \in K$ such that
\[
|x - z_\mathfrak{p}|_\mathfrak{p} < \varepsilon\ \hbox{for}\ \mathfrak{p} \in S,\ \hbox{and}\ |x|_\mathfrak{p} \le 1\ \hbox{for}\ \mathfrak{p} \not\in S \cup \{\mathfrak{p}_0\}.
\]
\end{lemma}

\begin{lemma}\label{residueSystemLemma}
For each $\varepsilon > 0$, there exists a complete set of representatives $\mathcal{D}^*$ of the residue class ring  $\mathfrak{z}^* / \alpha\, \mathfrak{z}^*$ satisfying $|d^*|_\mathfrak{p} \le \varepsilon$ for all $d^* \in \mathcal{D}^*$, $\mathfrak{p} \mid \mathfrak{b}$.
\end{lemma}

\begin{proof}
Lemma~\ref{strongapprox} implies that $\Phi_\mathfrak{b}(\mathcal{O}_{S_\alpha})$ is dense in~$\mathbb{K}_\mathfrak{b}$.
Since $\mathcal{O}_{S_\alpha} \subseteq \mathfrak{z}^*$, this yields that $\Phi_\mathfrak{b}({\xi + \alpha\, \mathfrak{z}^*})$ is also dense in~$\mathbb{K}_\mathfrak{b}$ for each $\xi \in \mathbb{Q}(\alpha)$.
Thus each residue class of $\mathfrak{z}^* / \alpha\, \mathfrak{z}^*$ contains a representative with the required property.
\end{proof}

In analogy to the notion of $\tau_\alpha$-invariance on~$\mathbb{K}_\alpha$, we call a set $Y_\infty \subseteq \mathbb{K}_\infty$ \emph{$\tau_\infty$-invariant} if, for each $\mathbf{z} \in Y_\infty$, $d^*
\in \mathcal{D}^*$, $u_\infty(\tau_{\infty,d^*}(\mathbf{z})) > 0$ implies $\tau_{\infty,d^*}(\mathbf{z}) \in Y_\infty$, where $u_\infty$ is defined by
\[
u_\infty:\ \mathbb{K}_\infty \to \mathbb{R}\,, \quad (z_\mathfrak{p})_{\mathfrak{p}\mid\infty} \mapsto \bigg|\frac{1}{|a_0|} \sum_{d \in \mathcal{D}} \prod_{\mathfrak{p}\mid\infty} \chi_\mathfrak{p}(d z_\mathfrak{p})\bigg|^2,
\]
and $\tau_{\infty,d^*}(\mathbf{z}) = \alpha^{-1} \cdot (\mathbf{z} + \Phi_\infty(d^*))$.
We now restrict our attention to the set $Z_{f_\infty}$.

\begin{lemma} \label{l:NCMIinfty}
Let $f$ be a special eigenfunction and suppose that $\mathcal{D}^*$ satisfies~\eqref{em2}.
Then there exists a non-empty compact minimal $\tau_\infty$-invariant set $Y_\infty \subset Z_{f_\infty}$.
\end{lemma}

\begin{proof}
By Lemmas~\ref{l:NCMI} and~\ref{l:YE}, there exists a non-empty compact minimal $\tau_\alpha$-invariant set $Y \subset Z_f \cap E$.
Arguing similarly as in the proof of Lemma~\ref{infinity}, we get
\begin{equation} \label{e:u}
u(\mathbf{z}) = u_\infty(\pi_\infty(\mathbf{z})) \quad \mbox{for each}\ \mathbf{z} \in E\,.
\end{equation}
Therefore, the $\tau_\alpha$-invariance of $Y$ implies $\tau_\infty$-invariance of $\pi_\infty(Y)$.
As in the proof of Lemma~\ref{l:NCMI}, this yields the existence of a non-empty compact minimal $\tau_\infty$-invariant set $Y_\infty \subseteq \pi_\infty(Y)$.
By Lemma~\ref{infinity}, we have $\pi_\infty(Y) \subseteq Z_{f_\infty}$, and the result is proved.
\end{proof}

\begin{lemma} \label{l:28}
Let $Y_\infty \subseteq Z_{f_\infty}$ be a non-empty compact minimal $\tau_\infty$-invariant set.
Then there exists a linear subspace $V$ of $\mathbb{K}_\infty \simeq \mathbb{R}^n$ such that
\begin{itemize}
\item
$\alpha \cdot V = V$,
\item
$Y_\infty$ is contained in a finite number of translates of $V$ and
\item
$Z_{f_\infty}$ contains a translate of~$V$.
\end{itemize}
\end{lemma}

\begin{proof}
This is a direct consequence of \cite[Theorem~2.8]{Cerveau-Conze-Raugi:96}, since we can identify $\mathbb{K}_\infty$ with $\mathbb{R}^n$, where multiplication by $\alpha$ is replaced by multiplication with the (block) diagonal matrix with entries $\alpha^{(\mathfrak{p})}$ if $K_\mathfrak{p} = \mathbb{R}$, and {\small$\begin{pmatrix}\mathrm{Re}(\alpha^{(\mathfrak{p})}) & \mathrm{Im}(\alpha^{(\mathfrak{p})}) \\ -\mathrm{Im}(\alpha^{(\mathfrak{p})}) & \mathrm{Re}(\alpha^{(\mathfrak{p})})\end{pmatrix}$} if $K_\mathfrak{p} = \mathbb{C}$, $\mathfrak{p} \mid \infty$.
\end{proof}

We have to exclude that $Y_\infty$ is finite, i.e., that $V = \{\mathbf{0}\}$ in Lemma~\ref{l:28}.

\begin{lemma}[{cf.\ \cite[Lemma~3.3]{Lagarias-Wang:97}}] \label{l:33}
There exists no non-empty finite $\tau_\infty$-invariant set $Y_\infty \subseteq Z_{f_\infty}$.
\end{lemma}

\begin{proof}
Suppose that there exists a non-empty finite $\tau_\infty$-invariant set $Y_\infty \subseteq Z_{f_\infty}$, and assume w.l.o.g. that $Y_\infty$ is minimal.
The minimality and finiteness imply that each $\mathbf{z} \in Y_\infty$ is contained in a ``$\tau_\infty$-cycle'', i.e., there exist $k \ge 1$ and $d_1^*,\ldots,d_k^* \in \mathcal{D}^*$ such that
\[
\tau_{\infty,d_1^*} \circ\cdots\circ \tau_{\infty,d_j^*}(\mathbf{z}) \in Y_\infty \qquad (1 \le j \le k)
\]
and $\tau_{\infty,d_1^*} \circ\cdots\circ \tau_{\infty,d_k^*}(\mathbf{z})=\mathbf{z}$.
Hence, we have $\mathbf{z} = \Phi_\infty(\xi)$ with $\xi = (1-\alpha^{-k})^{-1} \sum_{j=1}^k \alpha^{-j}  d_j^*$.
Since $|(1-\alpha^{-k})^{-1}|_\mathfrak{p} \le 1$ for each $\mathfrak{p}\mid\mathfrak{b}$, the choice of~$\mathcal{D}^*$ implies that $|\xi|_\mathfrak{p} \le |\alpha^{-m}|_\mathfrak{p}$.
Set
\[
\Xi = \Phi_\infty^{-1}(Y_\infty)\,.
\]
As $\Phi_\alpha(\Xi) \subset E$ and $\Phi_\infty(\Xi) = Y_\infty$, Lemma~\ref{infinity} and~\eqref{e:u} imply that $\Phi_\alpha(\Xi)$ is a finite minimal $\tau_\alpha$-invariant subset of~$Z_f$.

By Lemma~\ref{l:32}~(\ref{2}) and  the $\tau_\alpha$-invariance of $\Phi_\alpha(\Xi)$, we have $\Xi \subseteq \alpha \Xi - \mathcal{D}^*$.
Let $\bar{\Xi}$ be a set of representatives of $\Xi \bmod\, \mathfrak{z}^*$.
As $\mathcal{D}^* \subset \mathfrak{z}^*$, we get that $\bar{\Xi} \subseteq \alpha \bar{\Xi} \bmod\, \mathfrak{z}^*$.
The finiteness of~$\bar{\Xi}$ even yields that $\bar{\Xi} \equiv \alpha \bar{\Xi} \bmod\, \mathfrak{z}^*$, hence, $\alpha$ induces a permutation on~$\bar{\Xi}$.
The elements $\alpha^{-1} (\xi + d^*)$ are pairwise incongruent $\bmod\, \mathfrak{z}^*$ for different $d^*\in\mathcal{D}^*$.
Applying $\Phi_\alpha$ yields that for each $\xi \in \Xi$ there is a unique $d^* \in \mathcal{D}^*$ such that $\tau_{\alpha,d^*}(\Phi_\alpha(\xi)) \in \Phi_\alpha(\Xi)$, hence $u(\tau_{\alpha,d^*}(\Phi_\alpha(\xi))) = 1$.
Thus $u(\Phi_\alpha(\xi)) = 1$ holds for all $\xi \in \Xi$.
This yields that $\chi_\alpha(\Phi_\alpha(d\, \xi)) = 1$ for all $d \in \mathcal{D}$, and we conclude that $\chi_\alpha(\Phi_\alpha(x\, \xi)) = 1$ for all $x \in (\mathcal{D} - \mathcal{D}) \mathbb{Z}$.

Since $\bar{\Xi} \equiv \alpha \bar{\Xi} \bmod\, \mathfrak{z}^*$, we obtain that $\chi_\alpha(\Phi_\alpha(x\, \alpha\, \xi)) = 1$ for all $x \in (\mathcal{D} - \mathcal{D}) \mathbb{Z}$ and, inductively, $\chi_\alpha(\Phi_\alpha(x\, \xi)) = 1$ for all $x \in \mathbb{Z}\langle \alpha, \mathcal{D} \rangle = \mathfrak{z}$, i.e., $\xi \in \mathfrak{z}^*$.
Since $\Phi_\alpha(\xi) \in Z_f$ and $Z_f$ is $\Phi_\alpha(\mathfrak{z}^*)$-periodic, we get $\mathbf{0} \in Z_f$, which contradicts the assumptions on the zero set $Z_f$ of a special eigenfunction~$f$.
This contradiction proves the lemma.
\end{proof}

In view of Lemma~\ref{l:33}, we can now assume that Lemma~\ref{l:28} holds with $V \neq \{\mathbf{0}\}$.
Under this assumption, we get the following density result.

\begin{lemma} \label{l:42}
Let $V \neq \{\mathbf{0}\}$ be a linear subspace of $\mathbb{K}_\infty$ with $\alpha \cdot V = V$.
Then the set $V + \pi_\infty(\Phi_\alpha(\mathfrak{z}^*) \cap E)$ is dense in~$\mathbb{K}_\infty$.
\end{lemma}

\begin{proof}
Let $\mathbf{z} = (z_\mathfrak{p})_{\mathfrak{p}\mid\infty} \in V \setminus \{\mathbf{0}\}$.
Then we have $\mathbb{Q}(\alpha) \cdot \mathbf{z} \subset V$, and the denseness of $\Phi_\infty(\mathbb{Q}(\alpha))$ in~$\mathbb{K}_\infty$ implies that $V$ is dense in $\prod_{\mathfrak{p}\mid\infty:\,z_\mathfrak{p}\neq0} K_\mathfrak{p} \times \prod_{\mathfrak{p}\mid\infty:\,z_\mathfrak{p}=0} \{0\}$.
Choose $\mathfrak{p}_0 \mid \infty$ with $z_{\mathfrak{p}_0} \neq 0$.
The Strong Approximation Theorem (Lemma~\ref{strongapprox}) with $S = \{\mathfrak{p}:\, \mathfrak{p} \mid \infty\} \setminus \{\mathfrak{p}_0\}$ yields that the projection of $\Phi_\infty(\mathcal{O})$ to $\prod_{\mathfrak{p}\in S} K_\mathfrak{p}$ is dense.
This implies that $V + \Phi_\infty(\mathcal{O})$ is dense in~$\mathbb{K}_\infty$.
Choose $m$ as in~\eqref{em}.
Since $\alpha^{-m} \mathcal{O} \cap \mathcal{O}$ has finite index in~$\mathcal{O}$, we obtain that $V + \Phi_\infty(\alpha^{-m} \mathcal{O} \cap \mathcal{O})$ is also dense in~$\mathbb{K}_\infty$.
Observing that $\Phi_\alpha(\alpha^{-m} \mathcal{O} \cap \mathcal{O}) \subseteq \Phi_\alpha(\mathfrak{z}^*) \cap E$ proves the lemma.
\end{proof}

We are now in a position to finish the proof of Theorem~\ref{newtilingtheorem},
which states that $\{\mathcal{F} + \Phi_\alpha(x):\, x \in \mathfrak{z}\}$ forms a tiling of~$\mathbb{K}_\alpha$.

\begin{proof}[Proof of Theorem~\ref{newtilingtheorem}]
Suppose that $\{\mathcal{F} + \Phi_\alpha(x):\, x \in \mathfrak{z}\}$ does not form a tiling of~$\mathbb{K}_\alpha$.
By Lemmas~\ref{l:special}, \ref{residueSystemLemma}, \ref{l:NCMIinfty}, \ref{l:28}, and~\ref{l:33}, there exists a special eigenfunction $f \in \widehat\Omega$ and an $\alpha$-invariant linear subspace $V \neq \{\mathbf{0}\}$ of~$\mathbb{K}_\infty$ such that $Z_{f_\infty}$ contains a translate of~$V$.
Since $Z_f$ is $\Phi_\alpha(\mathfrak{z}^*)$-periodic, Lemma~\ref{infinity} implies that $Z_{f_\infty}$ is $\pi_\infty(\Phi_\alpha(\mathfrak{z}^*) \cap E)$-periodic.
Thus we may apply Lemma~\ref{l:42} in order to conclude that $Z_{f_\infty}$ is dense in~$\mathbb{K}_\infty$.
The continuity of~$f_\infty$ yields that $Z_{f_\infty} = \mathbb{K}_\infty$.

Now, we have to pull this back to~$\mathbb{K}_\alpha$.
Since $E \subset Z_f$ by Lemma~\ref{infinity}, the $\Phi_\alpha(\mathfrak{z}^*)$-periodicity of $Z_f$ yields that
\[
E + \Phi_\alpha(\mathcal{O}_{S_\alpha}) \subseteq E + \Phi_\alpha(\mathfrak{z}^*) \subset Z_f\,.
\]
The set $E + \Phi_\alpha(\mathcal{O}_{S_\alpha})$ is dense in~$\mathbb{K}_\alpha$ because $\mathbb{K}_\infty \times \Phi_\mathfrak{b}(\{0\}) \subset E$ and $\Phi_\mathfrak{b}(\mathcal{O}_{S_\alpha})$ is dense in~$\mathbb{K}_\mathfrak{b}$ (by the Strong Approximation Theorem).
Thus $f$ vanishes on a dense subset of~$\mathbb{K}_\alpha$.
The continuity of~$f$ now yields that $f \equiv 0$ on~$\mathbb{K}_\alpha$, contradicting the fact that $f$ is a special eigenfunction.
This proves the theorem.
\end{proof}

\section{intersective tiles and SRS tiles} \label{sec:relat-betw-tiles}

Now we consider the intersective tiles~$\mathcal{G}(x)$, $x \in \mathbb{Z}[\alpha]$, and show that for certain choices of~$\mathcal{D}$ they are intimately related to SRS tiles.

\subsection*{Tiling induced by intersective tiles}
As above, assume that $\alpha$ is an expanding algebraic number with minimal polynomial $A(X)$ given in \eqref{minpol}, $\mathcal{D}$~is a standard digit set for~$\alpha$, and set $\mathfrak{z} = \mathbb{Z}\langle \alpha, \mathcal{D}\rangle$.
First we prove~\eqref{e:slices}, showing that $\mathcal{F}$ can be built from intersective tiles.

\begin{proposition} \label{p:slices}
The intersective tiles~$\mathcal{G}(x)$, $x \in \mathfrak{z}$, form ``slices'' of the rational self-affine tile $\mathcal{F}$ in the sense that
\[
\mathcal{F} = \overline{\bigcup_{x \in \mathfrak{z}} \big(\mathcal{G}(x) - \Phi_\alpha(x)\big)}\,.
\]
\end{proposition}

\begin{proof}
Note that
\[
\mathcal{G}(x) - \Phi_\alpha(x) = \mathcal{F} \cap \big(\mathbb{K}_\infty \times \Phi_\mathfrak{b}(-x)\big)
\]
for all $x \in \mathfrak{z}$.
The set $\Phi_\mathfrak{b}(\mathfrak{z})$ is dense in~$\mathbb{K}_\mathfrak{b}$ by the Strong Approximation Theorem (Lemma~\ref{strongapprox}) and Lemma~\ref{l:subgroup}.
Since $\mathcal{F}$ is the closure of its interior by Theorem~\ref{basicProperties}~(\ref{property2}), the result follows.
\end{proof}

The sets $\mathcal{G}(x)$ can be characterized in terms of the $\mathbb{Z}$-module
\[
\Lambda_{\alpha,m} = \mathbb{Z}[\alpha] \cap \alpha^{m-1} \mathbb{Z}[\alpha^{-1}]\,.
\]
Here and in the following, $m$ is chosen in a way that $\mathcal{D} \subset \alpha^{m}\mathbb{Z}[\alpha^{-1}]$ holds.

\begin{lemma} \label{l:G}
For every $x \in \mathbb{Z}[\alpha]$ we have
\[
\mathcal{G}(x) = \Phi_\alpha(x) + \bigg\{\sum_{j=1}^\infty \Phi_\alpha(d_j \alpha^{-j}):\, d_j \in \mathcal{D},\,  \alpha^k x + \sum_{j=1}^k d_j \alpha^{k-j} \in \Lambda_{\alpha,m}\ \hbox{for all}\ k\ge 0\bigg\}\,.
\]
In particular, we have $\mathcal{G}(x) = \emptyset$ for all $x \in \mathbb{Z}[\alpha] \setminus \Lambda_{\alpha,m}$.
\end{lemma}

In the proof of Lemma~\ref{l:G}, we use the following observation.

\begin{lemma} \label{l:boundedaway}
For every $x \in \mathbb{Z}[\alpha,\alpha^{-1}] \setminus \mathbb{Z}[\alpha^{-1}]$ we have $\Phi_\mathfrak{b}(x) \not\in \overline{\Phi_\mathfrak{b}(\mathbb{Z}[\alpha^{-1}])}$.
\end{lemma}

\begin{proof}
We first show that 
\begin{equation} \label{e:boundedaway}
\mathbb{Z}[\alpha,\alpha^{-1}] \cap \mathcal{O}[\alpha^{-1}] \subseteq \alpha^h\, \mathbb{Z}[\alpha^{-1}]
\end{equation} 
for some integer $h \ge 0$.
Indeed, by analogous reasoning as in the proof of Lemma~\ref{l:subgroup}~(ii), $\mathbb{Z}[\alpha^{-1}]$~is a subgroup of finite index of~$\mathcal{O}[\alpha^{-1}]$.
Let $x_1,\ldots, x_\ell \in \mathcal{O}[\alpha^{-1}]$ be a complete set of representatives of $\mathcal{O}[\alpha^{-1}] / \mathbb{Z}[\alpha^{-1}]$, and choose integers $h_1,\ldots, h_\ell$ as follows. 
If $x_i \not\in \mathbb{Z}[\alpha,\alpha^{-1}]$, then set $h_i = 0$. 
If $x_i \in \mathbb{Z}[\alpha,\alpha^{-1}]$, then choose $h_i \ge 0$ in a way that $x_i \in \alpha^{h_i} \mathbb{Z}[\alpha^{-1}]$. 
As $x_i \not\in \mathbb{Z}[\alpha,\alpha^{-1}]$ implies $\mathbb{Z}[\alpha,\alpha^{-1}] \cap (\mathbb{Z}[\alpha^{-1}] + x_i) = \emptyset$, and $x_i \in \alpha^{h_i} \mathbb{Z}[\alpha^{-1}]$ implies $\mathbb{Z}[\alpha^{-1}] + x_i \subseteq \alpha^{h_i} \mathbb{Z}[\alpha^{-1}]$, we obtain that
\[
\mathbb{Z}[\alpha, \alpha^{-1}] \cap \mathcal{O}[\alpha^{-1}] = \bigcup_{i=1}^\ell \big(\mathbb{Z}[\alpha,\alpha^{-1}] \cap (\mathbb{Z}[\alpha^{-1}] + x_i)\big) \subseteq \alpha^{\max\{h_1,\ldots,h_\ell\}} \mathbb{Z}[\alpha^{-1}].
\]
Hence, \eqref{e:boundedaway} holds with $h = \max\{h_1,\ldots,h_\ell\}$.

To prove the lemma let $x \in \mathbb{Z}[\alpha,\alpha^{-1}] \setminus \mathbb{Z}[\alpha^{-1}]$ and suppose on the contrary that $\Phi_\mathfrak{b}(x) \in \overline{\Phi_\mathfrak{b}(\mathbb{Z}[\alpha^{-1}])}$.
Then there is $y \in \mathbb{Z}[\alpha^{-1}]$ such that $|y - x|_\mathfrak{p} \le |\alpha^{-h}|_\mathfrak{p}$ for all $\mathfrak{p} \mid \mathfrak{b}$, with $h$ as above, i.e., $\alpha^h (y-x) \in \mathcal{O}[\alpha^{-1}]$, cf.\ Lemma~\ref{l:subgroup}~(\ref{OSalpha}).
By~\eqref{e:boundedaway}, this gives $y-x \in \mathbb{Z}[\alpha^{-1}]$, contradicting that $y \in \mathbb{Z}[\alpha^{-1}]$ and $x \notin \mathbb{Z}[\alpha^{-1}]$.
\end{proof}

\begin{proof}[Proof of Lemma~\ref{l:G}]
Let $\mathbf{z} \in \mathcal{F} + \Phi_\alpha(x)$, i.e., $\mathbf{z} = \Phi_\alpha(x) + \sum_{j=1}^\infty \Phi_\alpha(d_j \alpha^{-j})$ with $d_j \in \mathcal{D}$.
We clearly have $\alpha^k x + \sum_{j=1}^k d_j \alpha^{k-j} \in \mathbb{Z}[\alpha]$ for all $k \ge 0$.

If $\alpha^k x + \sum_{j=1}^k d_j \alpha^{k-j} \in \Lambda_{\alpha,m}$ for all $k \ge 0$, then $x + \sum_{j=1}^k d_j \alpha^{-j} \in \alpha^{m-k-1} \mathbb{Z}[\alpha^{-1}]$, thus $|x + \sum_{j=1}^\infty d_j \alpha^{-j}|_\mathfrak{p} = 0$ for all $\mathfrak{p} \mid \mathfrak{b}$, i.e., $\mathbf{z} \in \mathbb{K}_\infty \times \prod_{\mathfrak{p}\mid \mathfrak{b}}\{0\}$.
This gives $\mathbf{z} \in \mathcal{G}(x)$.

If $\alpha^k x + \sum_{j=1}^k d_j \alpha^{k-j} \not\in \alpha^{m-1} \mathbb{Z}[\alpha^{-1}]$ for some $k \ge 0$, then Lemma~\ref{l:boundedaway} implies that $\Phi_\mathfrak{b}(\alpha^k x + \sum_{j=1}^k d_j \alpha^{k-j}) \not\in \overline{\Phi_\mathfrak{b}(\alpha^{m-1}\,\mathbb{Z}[\alpha^{-1}])}$.
Since $\sum_{j={k+1}}^\infty \Phi_\mathfrak{b}(d_j \alpha^{k-j}) \in \overline{\Phi_\mathfrak{b}(\alpha^{m-1}\,\mathbb{Z}[\alpha^{-1}])}$ by the choice of~$m$, we obtain $\pi_\mathfrak{b}(\alpha^k \cdot \mathbf{z}) \not\in \overline{\Phi_\mathfrak{b}(\alpha^{m-1}\,\mathbb{Z}[\alpha^{-1}])}$, in particular $\pi_\mathfrak{b}(\mathbf{z}) \neq \Phi_\mathfrak{b}(0)$, i.e., $\mathbf{z} \not\in \mathcal{G}(x)$.
\end{proof}

Define the map
\[
T_\alpha:\, \mathbb{Z}[\alpha] \to \mathbb{Z}[\alpha]\,, \quad x \mapsto \alpha^{-1} (x-d)\,,
\]
where $d$ is the unique digit in $\mathcal{D}$ such that $\alpha^{-1} (x-d) \in \mathbb{Z}[\alpha]$.

\begin{lemma} \label{l:betaxd}
For each $x \in \Lambda_{\alpha,m}$, we have $T_\alpha(x) \in \Lambda_{\alpha,m}$.
\end{lemma}

\begin{proof}
We have $T_\alpha(x) \in \mathbb{Z}[\alpha]$ by definition, and $T_\alpha(x) = \alpha^{-1} (x - d) \in \alpha^{m-1} \mathbb{Z}[\alpha^{-1}]$.
\end{proof}

After these preparations, we give the following set equations for intersective tiles.

\begin{proposition} \label{p:seteqG}
For each $x \in \mathbb{Z}[\alpha]$, we have
\[
\mathcal{G}(x) = \bigcup_{y \in T_\alpha^{-1}(x)} \alpha^{-1} \cdot \mathcal{G}(y)\,.
\]
\end{proposition}

\begin{proof}
This is a direct consequence of Lemmas~\ref{l:G} and~\ref{l:betaxd}.
\end{proof}

We now start the preparations for the proof of Theorems~\ref{srstiling} and~\ref{t:24}.
Recall that we identify $\mathbb{K}_\infty$ with $\mathbb{K}_\infty \times \Phi_\mathfrak{b}(\{0\})$.

\begin{lemma} \label{l:cover}
The collection $\{\mathcal{G}(x):\, x \in \mathfrak{z}\}$ forms a covering of~$\mathbb{K}_\infty$.
\end{lemma}

\begin{proof}
This follows from the fact that $\{\mathcal{F} + \Phi_\alpha(x):\, x \in \mathfrak{z}\}$ covers~$\mathbb{K}_\alpha$.
\end{proof}

In the next step, we consider properties of the $\mathbb{Z}$-module~$\Lambda_{\alpha,m}$. Recall that a subset of $\mathbb{K}_\infty$ is a lattice if it is a Delone set and an additive group.

\begin{lemma} \label{l:zLambda}
Let $k\in \mathbb{Z}$ be given. Then the sets $\Phi_\infty(\Lambda_{\alpha,k})$ and $\Phi_\infty(\mathfrak{z} \cap \Lambda_{\alpha,k})$ form lattices in~$\mathbb{K}_\infty \simeq \mathbb{R}^n$. Moreover,  the cardinality of $\Lambda_{\alpha,k+1} / \Lambda_{\alpha,k}$ is~$|a_n|$, where $a_n$ is the leading coefficient of the polynomial $A(X)$ in \eqref{minpol}.
\end{lemma}

\begin{proof}
For each $x \in \Lambda_{\alpha,k}$, we have $|x|_\mathfrak{p} \le |\alpha^{k-1}|_\mathfrak{p}$ for all $\mathfrak{p} \mid \mathfrak{b}$, and $|x|_\mathfrak{p} \le 1$ for all other $\mathfrak{p} \nmid \infty$.
Therefore, $\Phi_\infty(\Lambda_{\alpha,k})$ is a lattice and, hence, $\Phi_\infty(\mathfrak{z} \cap \Lambda_{\alpha,k})$ is contained in a lattice.
To show that $\Phi_\infty(\mathfrak{z} \cap \Lambda_{\alpha,k})$ contains a lattice, choose $x_1, \ldots, x_n \in \mathfrak{z}$ such that $\{\Phi_\infty(x_1), \ldots, \Phi_\infty(x_n)\}$ is a basis of~$\mathbb{K}_\infty$ (regarded as a vector space over~$\mathbb{R}$).
There exists $j \in \mathbb{N}$ such that, with $a_n$ as in \eqref{minpol}, 
we have  $a_n^j x_1, \ldots, a_n^j x_n \in \Lambda_{\alpha,k}$.
Since $\Phi_\infty(\mathfrak{z} \cap \Lambda_{\alpha,k})$ is a $\mathbb{Z}$-module containing the basis $\{\Phi_\infty(a_n^j x_1), \ldots, \Phi_\infty(a_n^j x_n)\}$ of~$\mathbb{K}_\infty$, it contains a lattice. Being an additive group, $\Phi_\infty(\mathfrak{z} \cap \Lambda_{\alpha,k})$ is therefore a lattice.

To show that the cardinality $\Lambda_{\alpha,k+1} / \Lambda_{\alpha,k}$ is $|a_n|$, first note that
\[
\Lambda_{\alpha,k+1} = \alpha^k \mathbb{Z}[\alpha^{-1}] \cap \mathbb{Z}[\alpha] = \bigcup_{c\in\{0,1,\ldots,|a_n|-1\}} (c\, \alpha^k + \alpha^{k-1} \mathbb{Z}[\alpha^{-1}]) \cap \mathbb{Z}[\alpha]
\]
is a partition of~$\Lambda_{\alpha,k+1}$.
For any $x \in (c\, \alpha^k + \alpha^{k-1} \mathbb{Z}[\alpha^{-1}]) \cap \mathbb{Z}[\alpha]$, we have 
\[
(c\, \alpha^k + \alpha^{k-1} \mathbb{Z}[\alpha^{-1}]) \cap \mathbb{Z}[\alpha] = x  + \Lambda_{\alpha,k}\,.
\]
It remains to show that $(c\, \alpha^k + \alpha^{k-1} \mathbb{Z}[\alpha^{-1}]) \cap \mathbb{Z}[\alpha] \neq \emptyset$ for all $c\in\{0,1,\ldots,|a_n|-1\}$.
As $\Phi_\mathfrak{b}(\mathbb{Z}[\alpha])$ is dense in~$\mathbb{K}_\mathfrak{b}$ by Lemma~\ref{l:subgroup}~(\ref{Zalphasubroup}) and the Strong Approximation Theorem in Lemma~\ref{strongapprox}, there is, for each $c \in \mathbb{Z}$, some $x \in \mathbb{Z}[\alpha]$ such that $\Phi_\mathfrak{b}(x) \in \overline{\Phi_\mathfrak{b}(c\, \alpha^k+\alpha^{k-1} \mathbb{Z}[\alpha^{-1}])}$.
By Lemma~\ref{l:boundedaway}, we have $x - c\, \alpha^k \in \alpha^{k-1} \mathbb{Z}[\alpha^{-1}]$, thus $x \in (c\, \alpha^k + \alpha^{k-1} \mathbb{Z}[\alpha^{-1}]) \cap \mathbb{Z}[\alpha] \neq \emptyset$.
\end{proof}

We are now in a position to show that $\mu_\infty(\partial\mathcal{G}(x)) = 0$ holds
for each $x \in \mathbb{Z}[\alpha]$. 

\begin{proof}[Proof of Theorem~\ref{srstiling}~(\ref{three:i})]
Let $x \in \mathbb{Z}[\alpha]$, and $X \subset \mathbb{K}_\infty$ be a rectangle containing~$\mathcal{G}(x)$. 
Since $\alpha^\ell \cdot \mathcal{G}(x) = \bigcup_{y \in T_\alpha^{-\ell}(x)} \mathcal{G}(y)$ by Proposition~\ref{p:seteqG}, we have $\alpha^\ell \cdot \partial\mathcal{G}(x) \subset \bigcup_{y\in\mathbb{Z}[\alpha]} \partial\mathcal{G}(y)$ and thus
\begin{equation} \label{e:muXB}
\frac{\mu_\infty(\partial\mathcal{G}(x))}{\mu_\infty(X)} = \frac{\mu_\infty(\alpha^\ell\cdot\partial\mathcal{G}(x))}{\mu_\infty(\alpha^\ell\cdot X)} \le \frac{\mu_\infty\big(\bigcup_{y\in\mathbb{Z}[\alpha]} \partial\mathcal{G}(y) \cap \alpha^\ell \cdot X\big)}{\mu_\infty(\alpha^\ell \cdot X)}
\end{equation}
for all $\ell \in \mathbb{N}$. 
To get an upper bound for $\mu_\infty\big(\bigcup_{y\in\mathbb{Z}[\alpha]} \partial\mathcal{G}(y) \cap \alpha^\ell \cdot X\big)$, let 
\[
R_k = \big\{z \in \mathcal{D}_k: \partial(\alpha^k \cdot \mathcal{F}) \cap \big(\mathcal{F} + \Phi_\alpha(z)\big) \neq \emptyset\big\}\,
\]
with $\mathcal{D}_k$ as in \eqref{iteratedset}.
Then, for every $y \in \mathbb{Z}[\alpha]$,
\[
\partial\mathcal{G}(y) \subset \partial\big(\mathcal{F} +  \Phi_\alpha(y)\big) \cap (\mathbb{K}_\infty \times \Phi_\mathfrak{b}(\{0\})) \subset \bigcup_{z\in R_k} \alpha^{-k} \cdot \big(\mathcal{F} + \Phi_\alpha(\alpha^k\, y+z)\big) \cap (\mathbb{K}_\infty \times \Phi_\mathfrak{b}(\{0\})).
\]
If $\Phi_\mathfrak{b}(0) \in \pi_\mathfrak{b}(\mathcal{F}) + \Phi_\mathfrak{b}(\alpha^k\, y+z) \subset  \overline{\Phi_\mathfrak{b}(\alpha^{m-1}\,\mathbb{Z}[\alpha^{-1}])} + \Phi_\mathfrak{b}(\alpha^k\, y+z)$ holds for a given $z \in R_k$, then Lemma~\ref{l:boundedaway} yields $\alpha^k\, y + z \in \alpha^{m-1}\,\mathbb{Z}[\alpha^{-1}]$. Thus
\begin{equation}\label{Ginclusion}
\bigcup_{y\in\mathbb{Z}[\alpha]} \partial\mathcal{G}(y) \subset \bigcup_{z\in R_k} \bigcup_{y\in \mathbb{Z}[\alpha] \cap \alpha^{-k} (\alpha^{m-1}\mathbb{Z}[\alpha^{-1}]-z)} \Phi_\infty(y + \alpha^{-k} z) + \pi_\infty(\alpha^{-k} \cdot \mathcal{F})\,.
\end{equation}
If we set
\[
C_{k,\ell}(z) = \{y \in \mathbb{Z}[\alpha] \cap \alpha^{-k} (\alpha^{m-1}\mathbb{Z}[\alpha^{-1}]-z):\, \Phi_\infty(y + \alpha^{-k} z) \in \alpha^\ell \cdot X - \pi_\infty(\alpha^{-k} \cdot \mathcal{F})\big\}\,,
\]
then \eqref{Ginclusion} implies that
\begin{equation}\label{Ginclusion2}
\bigcup_{y\in\mathbb{Z}[\alpha]} \partial\mathcal{G}(y) \cap \alpha^\ell \cdot X
\subset
\bigcup_{z\in R_k} \bigcup_{y\in C_{k,\ell}(z)} \Phi_\infty(y + \alpha^{-k} z) + \pi_\infty(\alpha^{-k} \cdot \mathcal{F})\,.
\end{equation}
To estimate the measure of the right hand side of~\eqref{Ginclusion2}, we first deal with the number of elements in $R_k$ and $C_{k,\ell}(z)$. From the proof of Theorem~\ref{basicProperties}~(\ref{property3}), we get that 
\begin{equation}\label{Gestimate1}
\# R_k = O(r^k)
\end{equation}
with $r < |a_0|$.
To derive an estimate for $\# C_{k,\ell}(z)$, observe that, for each $z \in \mathcal{D}_k$, the set $\mathbb{Z}[\alpha] \cap \alpha^{-k} (\alpha^{m-1}\mathbb{Z}[\alpha^{-1}]-z)$ forms a residue class of $\Lambda_{\alpha,m} / \Lambda_{\alpha,m-k}$.
By Lemma~\ref{l:zLambda}, the cardinality of $\Lambda_{\alpha,m} / \Lambda_{\alpha,m-k}$ is~$|a_n|^k$.
For $\ell$ sufficiently large (in terms of~$k$), we obtain that 
\begin{equation}\label{Gestimate2}
\frac{\max_{z\in R_k} \# C_{k,\ell}(z)}{\# \big\{y \in \Lambda_{\alpha,m}:\, \Phi_\infty(y) \in \alpha^\ell \cdot X\big\}} \le \frac{2}{|a_n|^k}\,.
\end{equation}
(Just note that $\Phi_\infty(C_{k,\ell}(z))$ is essentially the intersection of a shifted version of the sublattice $\Phi_\infty(\Lambda_{\alpha, m-k})$ of $\Phi_\infty(\Lambda_{\alpha, m})$ with the large rectangle $\alpha^\ell\cdot X$; the subtraction of $\pi_\infty(\alpha^{-k} \cdot \mathcal{F})$ is minor as $\ell$ is large.) 
As $\Phi_\infty(\Lambda_{\alpha,m})$ is a lattice in $\mathbb{K}_\infty$, we gain that 
\begin{equation}\label{Gestimate3}
\# \{y \in \Lambda_{\alpha,m}:\, \Phi_\infty(y) \in \alpha^\ell \cdot X\} = O\big(\mu_\infty(\alpha^\ell \cdot X)\big)\,.
\end{equation}
Since $\prod_{\mathfrak{p}\mid\infty} |\alpha|_\mathfrak{p} = \frac{|a_0|}{|a_n|}$, we have
\begin{equation}\label{Gestimate4}
\mu_\infty\big(\pi_\infty(\alpha^{-k} \cdot \mathcal{F})\big) = O\bigg(\frac{|a_n|^k}{|a_0|^k}\bigg)\,.
\end{equation}
Putting \eqref{Ginclusion2} together with the estimates \eqref{Gestimate1},  \eqref{Gestimate2}, \eqref{Gestimate3}, and~\eqref{Gestimate4}, we arrive at
\[
\frac{\mu_\infty\big(\bigcup_{y\in\mathbb{Z}[\alpha]} \partial\mathcal{G}(y) \cap \alpha^\ell \cdot X\big)}{\mu_\infty(\alpha^\ell \cdot X)} \le \frac{\# R_k\, \max_{z\in R_k} \# C_{k,\ell}(z)\, \mu_\infty\big(\pi_\infty(\alpha^{-k} \cdot \mathcal{F})\big)}{\mu_\infty(\alpha^\ell \cdot X)} \le  \frac{c\,r^k}{|a_0|^k}
\]
for some constant $c > 0$, which holds for each $k \in \mathbb{N}$ and for all sufficiently large~$\ell$. 
Inserting this estimate in \eqref{e:muXB} implies that $\frac{\mu_\infty(\partial\mathcal{G}(x))}{\mu_\infty(X)} \le \frac{c\,r^k}{|a_0|^k}$ for all $k \in \mathbb{N}$, i.e., $\mu_\infty(\partial\mathcal{G}(x)) = 0$. 
\end{proof}

The main problem in the proof of the tiling property of the collection $\{\mathcal{G}(x):\, x \in \mathfrak{z}\}$ consists in finding an \emph{exclusive point}, i.e., a point $\mathbf{z} \in \mathbb{K}_\infty$ which is contained in exactly one element of this collection.
Using the tiling theorem for rational self-affine tiles (Theorem~\ref{newtilingtheorem}) we exhibit such an exclusive point in the two following lemmas (see \cite[Section~4]{BSSST:11}, where similar methods were employed).

\begin{lemma} \label{l:Takz}
Let $Y = \{y \in \mathfrak{z} \cap \Lambda_{\alpha,m}:\, \Phi_\infty(y) \in \pi_\infty(\mathcal{F})\}$.
There exist $z \in \mathfrak{z} \cap \Lambda_{\alpha,m}$, $k \ge 0$, such~that
\[
T_\alpha^k(z + y) = 0 = T_\alpha^k(z) \quad\mbox{for all}\ y \in Y.
\]
\end{lemma}

\begin{proof}
By Lemma~\ref{l:zLambda} and because $\pi_\infty(\mathcal{F})$ is compact, the set $Y$ is finite.
Since the interior of~$\mathcal{F}$ is non-empty by Theorem~\ref{basicProperties}~(\ref{property2}), there exists an open ball $B \subset \mathrm{int}(\mathcal{F})$.
By Lemma~\ref{l:zLambda}, we can find some $k \ge 0$ and some $z \in \mathfrak{z} \cap \Lambda_{\alpha,m}$ such that $z + Y \subset \alpha^k \cdot B$.
We have $\alpha^k \cdot \mathcal{F} = \mathcal{F} + \mathcal{D}_k$ by~\eqref{iteratedset}, and $\mathfrak{z} \cap \alpha^k \cdot \mathrm{int}(\mathcal{F}) \subseteq \mathcal{D}_k$ by Theorem~\ref{newtilingtheorem}.
Hence, we conclude that $z + Y \subseteq \mathcal{D}_k$, i.e., $T_\alpha^k(z+y) = 0$ for all $y \in Y$.
\end{proof}

\begin{lemma} \label{l:exclusive}
Let $z \in \mathfrak{z} \cap \Lambda_{\alpha,m}$, $k \ge 0$ such that $T_\alpha^k(z + y) = T_\alpha^k(z)$ for all $y \in Y$, with $Y$ as in Lemma~\ref{l:Takz}.
Then $\Phi_\infty(\alpha^{-k} z) \in \mathcal{G}(T_\alpha^k(z))$ and $\Phi_\infty(\alpha^{-k} z) \not\in \mathcal{G}(x)$ for all $x \in \mathfrak{z} \setminus \{T_\alpha^k(z)\}$, i.e., $\Phi_\infty(\alpha^{-k} z)$ is an exclusive point of~$\mathcal{G}(T_\alpha^k(z))$.
\end{lemma}

\begin{proof}
Consider any $x \in \mathfrak{z}$ such that $\Phi_\infty(\alpha^{-k} z) \in \mathcal{G}(x)$.
Note that such an $x$ exists by Lemma~\ref{l:cover}.
By Proposition~\ref{p:seteqG}, there exists $z' \in T_\alpha^{-k}(x) \cap \Lambda_{\alpha,m}$ such that $\Phi_\infty(\alpha^{-k} z) \in \alpha^{-k} \cdot \mathcal{G}(z')$.
This means that $\Phi_\infty(\alpha^{-k} z) = \alpha^{-k} \cdot \big(\Phi_\infty(z') + \sum_{j=1}^\infty \Phi_\infty(d_j \alpha^{-j})\big)$ for some $d_j \in \mathcal{D}$, thus $\Phi_\infty\big(z - z') \in \pi_\infty(\mathcal{F})$.
Since $T_\alpha^{-k}(\mathfrak{z}) \subseteq \mathfrak{z}$, we have $z - z' \in Y$.
By the definition of~$z'$ and the assumption of the lemma, we obtain that
\[
x = T_\alpha^k(z') = T_\alpha^k(z + (z'-z)) = T_\alpha^k(z)\,.
\]
Therefore, $\Phi_\infty(\alpha^{-k} z)$ is an exclusive point of $\mathcal{G}(T_\alpha^k(z))$.
\end{proof}

We can now show that $\{\mathcal{G}(x):\, x \in \mathfrak{z}\}$ forms a tiling of~$\mathbb{K}_\infty$.

\begin{proof}[Proof of Theorem~\ref{srstiling}~(\ref{three:ii})]
Being the intersection of the compact set $\mathcal{F} + \Phi_\alpha(x)$ with $\mathbb{K}_\infty \times \Phi_\mathfrak{b}(\{0\})$, the set $\mathcal{G}(x)$~is compact for each $x \in \mathfrak{z}$.
The collection $\{\mathcal{G}(x):\, x \in \mathfrak{z}\}$ is uniformly locally finite because $\mathcal{G}(x) = \emptyset$ for all $x \in \mathfrak{z} \setminus \Lambda_{\alpha,m}$ by Lemma~\ref{l:G}, $\Phi_\infty(\mathfrak{z} \cap \Lambda_{\alpha,m})$ forms a lattice of $\mathbb{K}_\infty$ by Lemma~\ref{l:zLambda}, and $\mathcal{G}(x) - \Phi_\infty(x) \subseteq \pi_\infty(\mathcal{F})$ for all $x \in \mathfrak{z}$.

Let $z \in \mathfrak{z} \cap \Lambda_{\alpha,m}$ and $k \ge 0$ be as in Lemma~\ref{l:Takz}.
By the definition of~$T_\alpha$, this implies that $T_\alpha^k(x + y) = T_\alpha^k(x)$ for all $x \in z + \alpha^k \mathbb{Z}[\alpha]$, $y \in Y$.
In particular, this holds true for $x \in z + \alpha^k (\mathfrak{z} \cap \Lambda_{\alpha,m-k}) \subseteq \mathfrak{z} \cap \Lambda_{\alpha,m}$.
By Lemma~\ref{l:exclusive}, the point $\Phi_\infty(\alpha^{-k} x)$ is exclusive for each of these~$x$.
Since $\Phi_\infty(\mathfrak{z} \cap \Lambda_{\alpha,m-k})$ is a lattice, we have found a relatively dense set of exclusive points in~$\mathbb{K}_\infty$.
Proposition~\ref{p:seteqG} shows that, for every exclusive point~$\mathbf{z}$, the point $\alpha^{-1} \cdot \mathbf{z}$ is exclusive as well.
Therefore, the set of exclusive points is dense.
Since the boundary of each tile has zero measure (by Theorem~\ref{srstiling}~(\ref{three:i})) and $\{\mathcal{G}(x):\, x \in \mathfrak{z}\}$ is a uniformly locally finite collection of compact sets, this proves the tiling property.
\end{proof}

For certain digit sets, the set $\mathcal{G}(x)$ is non-empty for each $x$ in the $\mathbb{Z}$-module~$\Lambda_{\alpha,m}$. 
This is made precise in the following lemma.

\begin{lemma} \label{l:continuation}
Suppose that $\mathcal{D}$ contains a complete residue system of $\alpha^m \mathbb{Z}[\alpha^{-1}] / \alpha^{m-1} \mathbb{Z}[\alpha^{-1}]$.
Then, for every $x \in \Lambda_{\alpha,m}$, we have $\alpha x + d \in \Lambda_{\alpha,m}$ for some $d \in \mathcal{D}$, thus $\mathcal{G}(x) \neq \emptyset$.
\end{lemma}

\begin{proof}
For each $x \in \Lambda_{\alpha,m}$, we have $\alpha x + d \in \mathbb{Z}[\alpha] \cap \alpha^m \mathbb{Z}[\alpha^{-1}]$ for all $d \in \mathcal{D}$.
Since $\mathcal{D}$ contains a complete residue system of $\alpha^m \mathbb{Z}[\alpha^{-1}] / \alpha^{m-1} \mathbb{Z}[\alpha^{-1}]$, there exists some $d \in \mathcal{D}$ such that $\alpha x + d \in \alpha^{m-1} \mathbb{Z}[\alpha^{-1}]$.
Inductively, we obtain a sequence $d_1, d_2, \ldots$ such that $\alpha^k x + \sum_{j=1}^k d_j \alpha^{k-j} \in \Lambda_{\alpha,m}$ for all $k\ge 0$, thus $\Phi_\alpha(x) + \sum_{j=1}^\infty \Phi_\alpha(d_j \alpha^{-j}) \in \mathcal{G}(x)$.
\end{proof}

To prepare the proof of Theorem~\ref{t:24}~(\ref{almostperiodic}), we start with the following representation of~$\mathcal{G}(x)$.

\begin{lemma} \label{l:GLim}
Suppose that $\mathcal{D}$ contains a complete residue system of $\alpha^m \mathbb{Z}[\alpha^{-1}] / \alpha^{m-1} \mathbb{Z}[\alpha^{-1}]$.
Then, for each $x \in \Lambda_{\alpha,m}$, we have
\[
\mathcal{G}(x) = \mathop{\mathrm{Lim}}_{k\to\infty} \Phi_\infty\big(\alpha^{-k} \big(T_\alpha^{-k}(x) \cap \Lambda_{\alpha,m}\big)\big)\,,
\]
where the limit is taken with the respect to the Hausdorff distance~$\delta_H$.
\end{lemma}

\begin{proof}
Let $x \in \Lambda_{\alpha,m}$.
By Proposition~\ref{p:seteqG} and Lemma~\ref{l:G}, we have
\[
\mathcal{G}(x) = \bigcup_{y \in T_\alpha^{-k}(x) \cap \Lambda_{\alpha,m}} \alpha^{-k} \cdot \mathcal{G}(y)\,.
\]
All the involved sets are non-empty by Lemma~\ref{l:continuation}, and
\[
\max_{y \in \Lambda_{\alpha,m}} \mathrm{diam}\big(\alpha^{-k} \cdot \mathcal{G}(y)\big) \le c' \max_{\mathfrak{p}\mid\infty} |\alpha^{-k}|_\mathfrak{p}
\]
for some $c' > 0$ because $\mathcal{G}(y) \subset \pi_\infty(\mathcal{F})$.
This implies that
\begin{equation} \label{e:deltaH}
\delta_H\Big(\mathcal{G}(x), \Phi_\infty\big(\alpha^{-k} \big(T_\alpha^{-k}(x) \cap \Lambda_{\alpha,m}\big)\big)\Big) \le c' \max_{\mathfrak{p}\mid\infty} |\alpha^{-k}|_\mathfrak{p}\,.
\end{equation}
Since $\alpha^{-1}$ is contracting, this yields the lemma.
\end{proof}

\begin{lemma} \label{l:xminusy}
Let $x, y \in \Lambda_{\alpha,m}$, $k \ge 0$ such that $x-y \in \Lambda_{\alpha,m-k}$.
Then
\[
\alpha^{-k} \big(T_\alpha^{-k}(x) \cap \Lambda_{\alpha,m}\big) - x = \alpha^{-k} \big(T_\alpha^{-k}(y) \cap \Lambda_{\alpha,m}\big) - y\,.
\]
\end{lemma}

\begin{proof}
Let $x, y \in \Lambda_{\alpha,m}$, $k \ge 0$ such that $x-y \in \Lambda_{\alpha,m-k}$.
This implies that $\alpha^k (x-y) \in \Lambda_{\alpha,m}$.
Therefore, for each $d \in \mathcal{D}_k$, $\alpha^k\, x + d \in \Lambda_{\alpha,m}$ is equivalent to $\alpha^k\, y + d \in \Lambda_{\alpha,m}$.
We get that
\begin{multline*}
\big(T_\alpha^{-k}(x) \cap \Lambda_{\alpha,m}\big) - \alpha^k\, x = \{d \in \mathcal{D}_k:\, \alpha^k\, x + d \in \Lambda_{\alpha,m}\} \\
= \{d \in \mathcal{D}_k:\, \alpha^k\, y + d \in \Lambda_{\alpha,m}\} = \big(T_\alpha^{-k}(y) \cap \Lambda_{\alpha,m}\big) - \alpha^k\, y\,.
\end{multline*}
Multiplying by~$\alpha^{-k}$, the lemma follows.
\end{proof}

\begin{proof}[{Proof of Theorem~\ref{t:24}}]
The fact that $\{\mathcal{G}(x):\, x \in \Lambda_{\alpha,m} \cap \mathbb{Z}\langle \alpha, \mathcal{D}\rangle\}$ forms a tiling of~$\mathbb{K}_\infty$ follows directly from Theorem~\ref{srstiling} and Lemma~\ref{l:G}.
Assertion (\ref{Gnotempty}) is the content of Lemmas~\ref{l:G} and~\ref{l:continuation}.
Finally, assertion (\ref{almostperiodic}) is a consequence of Lemma~\ref{l:xminusy} and~\eqref{e:deltaH}.
Indeed, we obtain that
\begin{align*}
\delta_H\big(\mathcal{G}(x) - \Phi_\infty(x), \mathcal{G}(y) - \Phi_\infty(y)\big) & \le \delta_H\Big(\mathcal{G}(x) - \Phi_\infty(x), \Phi_\infty\big(\alpha^{-k} \big(T_\alpha^{-k}(x) \cap \Lambda_{\alpha,m}\big) - x\big)\Big) \\
& \quad + \delta_H\Big(\Phi_\infty\big(\alpha^{-k} \big(T_\alpha^{-k}(y) \cap \Lambda_{\alpha,m}\big) - y\big), \mathcal{G}(y) - \Phi_\infty(y)\Big) \\
& \le 2 c' \max_{\mathfrak{p}\mid\infty} |\alpha^{-k}|_\mathfrak{p}\,. \qedhere
\end{align*}
\end{proof}

The following example shows that the set $\{x \in \mathbb{Z}[\alpha]:\, \mathcal{G}(x) \neq \emptyset\}$ need not be a $\mathbb{Z}$-module when the condition of Lemma~\ref{l:continuation} is not satisfied.

\begin{example}
Let $\alpha = \frac{4}{3}$ and $\mathcal{D} = \{0, 1, 2, \frac{1}{3}\} \subset \alpha \mathbb{Z}[\alpha^{-1}]$, which gives the tiling depicted in Figure~\ref{fig:m34}.
For this choice of $\alpha$ and~$\mathcal{D}$, \eqref{em} holds with~$m = 1$, and Lemma~\ref{l:G} implies that $\mathcal{G}(x) \ne \emptyset$ holds if and only if $x \in T_\alpha^k(\Lambda_{\alpha,1})$ for all $k \ge 0$.
We first observe that $\Lambda_{\alpha,1} = \mathbb{Z}[\tfrac{1}{3}] \cap \mathbb{Z}[\tfrac{1}{4}] = \mathbb{Z}$ and
\[
T_\alpha(\Lambda_{\alpha,1}) = \{x \in \mathbb{Z}:\, \alpha x + d \in \mathbb{Z}\ \mbox{for some}\ d \in \mathcal{D}\} = \{x \in \mathbb{Z}:\, x \not\equiv 1 \bmod 3\}\,.
\]
Inductively, we obtain that
\begin{align*}
T_\alpha^k(\Lambda_{\alpha,1}) & = \{x \in T_\alpha^{k-1}(\Lambda_{\alpha,1}):\, \alpha x + d \in T_\alpha^{k-1}(\Lambda_{\alpha,1})\ \mbox{for some}\ d \in \mathcal{D}\} \\
& = \{x \in \mathbb{Z}:\, x \not\equiv 2\,3^j-1 \bmod 3^{j+1}\ \mbox{for all}\ 0 \le j < k\}\,,
\end{align*}
e.g. $T_\alpha^4(\Lambda_{\alpha,1}) = \{0,2,3,6,8,9,11,12,15,18,20,21,24,26\} + 27\, \mathbb{Z}$.
Therefore, the set $\mathcal{G}(x)$, $x \in \mathbb{Z}[\alpha]$, is non-empty if and only if $x \in \mathbb{Z}$ and $x \not\equiv 2\,3^j-1 \bmod 3^{j+1}$ for all $j \ge 0$.

Since $T_\alpha^{-1}(-3) = \{-3\}$ and $T_\alpha^{-1}(-1) = \{-1\}$, we have $\mathcal{G}(-3) = \{0\} = \mathcal{G}(-1)$.

\begin{figure}[ht]
\includegraphics{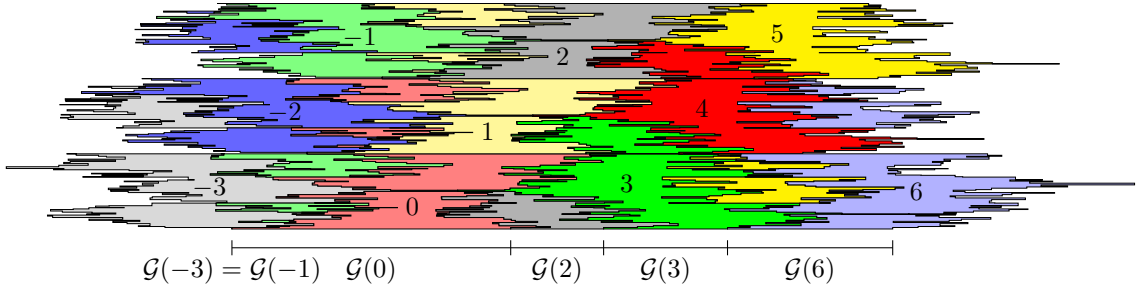}
\caption{The tiles $\mathcal{F} + \Phi_\alpha(x) \in \mathbb{R} \times \mathbb{Q}_3$ for $\alpha = \frac{4}{3}$, $\mathcal{D} = \{0,1,2,\frac{1}{3}\}$, $x \in \{-6,-5,\ldots,3\}$, and the corresponding intersective tiles~$\mathcal{G}(x) \in \mathbb{R}$.
An element $\sum_{j=k}^\infty b_j \alpha^{-j}$ of~$\mathbb{Q}_3$, with $b_j \in \{0,1,2\}$, is represented by $\sum_{j=k}^\infty b_j 3^{-j}$.} \label{fig:m34}
\end{figure}
\end{example}

\subsection*{SRS tiles}
We will now relate SRS tiles with intersective tiles given by digit sets of the shape $\mathcal{D} = \{0, 1, \ldots, |a_0|-1\}$.
This will enable us to infer that Theorem~\ref{srsthm} is a consequence of Theorem~\ref{t:24}.

Recall that the minimal polynomial $A(X) = a_n X^n + \cdots + a_1 X + a_0$ of $\alpha$ is a primitive expanding polynomial.
Therefore, we have $|a_0| > |a_n|$, and the set $\mathcal{D} = \{0, 1, \ldots, |a_0|-1\}$ contains a complete residue system of $\mathbb{Z}[\alpha^{-1}] / \alpha^{-1} \mathbb{Z}[\alpha^{-1}]$.
By Lemmas~\ref{l:G} and~\ref{l:continuation}, this implies that $\mathcal{G}(x)$ is non-empty if and only if $x \in \Lambda_{\alpha,0}$.
We first determine a basis of~$\Lambda_{\alpha,0}$.

\begin{lemma} \label{l:Lambda}
The $\mathbb{Z}$-module $\Lambda_{\alpha,0}$ is generated by $w_0 = a_n$, $w_i = \alpha w_{i-1} + a_{n-i} $, $1 \le i < n$.
\end{lemma}

\begin{proof}
For $0 \le i < n$, we have $w_i = \sum_{j=0}^i a_{n-i+j} \alpha^j = - \sum_{j=i-n}^{-1} a_{n-i+j} \alpha^j$, thus $w_i \in \Lambda_{\alpha,0}$ for $0 \le i < n$.
Clearly, the $\mathbb{Z}$-module generated by these elements is in $\Lambda_{\alpha,m}$ as well.

Let $x \in \Lambda_{\alpha,0}$, i.e., $x = P(\alpha)$ and $x = Q(\alpha)$ with polynomials $P \in \mathbb{Z}[X]$, $Q \in X^{-1} \mathbb{Z}[X^{-1}]$.
Therefore, $P(X) - Q(X)$ is a multiple of~$A(X)$ in $\mathbb{Z}[X,X^{-1}]$ and, hence, the leading coefficent $p$ of $P(X)$ is divisible by~$a_n$.
If $\mathrm{deg}(P) \ge n$, then $R(X) = P(X) - \frac{p}{a_n} X^{\mathrm{deg}(P)-n} A(X) \in \mathbb{Z}[X]$ gives the alternative representation $x = R(\alpha)$, where $\mathrm{deg}(R) < \mathrm{deg}(P)$.
Therefore, we can assume that $\mathrm{deg}(P) < n$., i.e., $x = \sum_{j=0}^{n-1} p_j \alpha^j$ with $p_j \in \mathbb{Z}$.

Set $P_1(X) = \sum_{j=0}^{n-1} p_j X^j$ and
\[
b_{n-1} = \frac{p_{n-1}}{a_n}\,,\ b_{n-2} = \frac{p_{n-2} - b_{n-1} a_{n-1}}{a_n}\,,\ \dots\,,\ b_0 = \frac{p_0 - b_{n-1} a_1 - \cdots - b_1 a_{n-1}}{a_n}\,.
\]
Define, recursively for $1 \le i < n$, the Laurent polynomials
\[
P_{i+1}(X) = P_i(X) - b_{n-i} X^{i-n} A(X)\,.
\]
Inductively we obtain that $a_n b_{n-i}$ is the coefficient of $X^{n-i}$ in~$P_i(X)$, and it is either $0$ or the leading coefficient of~$P_i(X)$.
Now, $P_i(\alpha) - Q(\alpha) = 0$ implies that $a_n b_{n-i}$ is divisible by~$a_n$, hence $b_{n-i}$ is an integer.
By the definition of~$w_i$, we have $\sum_{j=0}^{n-1} p_j \alpha^j = \sum_{i=0}^{n-1} b_i w_i$.
Therefore, $x$~is in the $\mathbb{Z}$-module generated by $w_0,\ldots,w_{n-1}$.
\end{proof}

In view of Lemma~\ref{l:Lambda}, we define the mapping
\[
\iota_\alpha:\ \mathbb{Q}^n \to \mathbb{Q}(\alpha)\,, \quad (z_0, \ldots, z_{n-1}) \mapsto \mathop{\mathrm{sgn}}(a_0) \sum_{i=0}^{n-1} z_i w_i\,.
\]
The following proposition is the transcription of \cite[Theorem~5.12]{BSSST:11} into our setting.
Since establishing all the necessary notational correspondences is more complicated than giving a proof, we include its proof.

\begin{proposition} \label{p:SRS}
Let $\mathcal{D} = \{0, 1, \ldots, |a_0|-1\}$, $\mathbf{r} = (\frac{a_n}{a_0},\ldots,\frac{a_1}{a_0})$.
Then we have
\[
\mathcal{G}\big(\iota_\alpha(\mathbf{z})\big) = \Upsilon\big(\mathcal{T}_\mathbf{r}(\mathbf{z})\big) \quad \mbox{for all}\ \mathbf{z} \in \mathbb{Z}^n,
\]
where $\Upsilon:\, \mathbb{R}^n \to \mathbb{K}_\infty$ is the linear transformation that is equal to $\Phi_\infty \circ \iota_\alpha$ on~$\mathbb{Q}^n$.
\end{proposition}

\begin{proof}
The matrix
\[
\mathbf{M}_\mathbf{r} = \begin{pmatrix}
0 & 1 & 0 & \cdots & 0 \\
\vdots & \ddots & \ddots & \ddots & \vdots \\
\vdots & & \ddots & \ddots & 0 \\
0 & \cdots & \cdots & 0 & 1 \\
\frac{-a_n}{a_0} & \frac{-a_{n-1}}{a_0} & \cdots & \frac{-a_2}{a_0} & \frac{-a_1}{a_0}
\end{pmatrix}
\]
represents the multiplication by $\alpha^{-1}$ with respect to the basis $\{w_0,w_1,\ldots,w_{n-1}\}$ of~$\mathbb{Q}(\alpha)$, which is defined in Lemma~\ref{l:Lambda}.
This means that $\alpha^{-1} \iota_\alpha(\mathbf{z}) = \iota_\alpha(M_\mathbf{r}\, \mathbf{z})$ for all $\mathbf{z} \in \mathbb{Q}^n$.
For $\mathbf{z} \in \mathbb{Z}^n$, we have, using the definition of~$\tau_\mathbf{r}$ and the fact that $w_{n-1} = - a_0 \alpha^{-1}$,
\begin{align*}
\iota_\alpha\big(\tau_\mathbf{r}(\mathbf{z})\big) & = \iota_\alpha\big(\mathbf{M}_\mathbf{r}\, \mathbf{z} + \big(0, \ldots, 0, \mathbf{r z} - \lfloor \mathbf{r z} \rfloor\big)\big) \\
& = \alpha^{-1} \iota_\alpha(\mathbf{z}) + \tfrac{d}{|a_0|} \mathop{\mathrm{sgn}}(a_0) w_{n-1} = \alpha^{-1} \big(\iota_\alpha(\mathbf{z}) - d\big) = T_\alpha\big(\iota_\alpha(\mathbf{z})\big)\,,
\end{align*}
where $d = |a_0|\, (\mathbf{r z} - \lfloor \mathbf{r z} \rfloor)$ is the unique element in $\mathcal{D}$ such that $\alpha^{-1} (\iota_\alpha(\mathbf{z}) - d) \in \mathbb{Z}[\alpha]$.
Iterating this and observing that $\mathbf{z} \in \mathbb{Z}^n$ is equivalent to $\iota_\alpha(\mathbf{z}) \in \Lambda_{\alpha,0}$, we obtain that
\[
\iota_\alpha\big(\mathbf{M}_\mathbf{r}^k\, \tau_\mathbf{r}^{-k}(\mathbf{z})\big) = \alpha^{-k}\, \Big(T_\alpha^{-k}\big(\iota_\alpha(\mathbf{z})\big) \cap \Lambda_{\alpha,0}\Big) \quad \mbox{for all}\ k \ge 0,
\]
and
\[
\Upsilon\big(\mathcal{T}_\mathbf{r}(\mathbf{z})\big) = \Upsilon\Big(\mathop{\mathrm{Lim}}_{k\to\infty} \mathbf{M}_\mathbf{r}^k\, \tau_\mathbf{r}^{-k}(\mathbf{z})\Big) = \mathop{\mathrm{Lim}}_{k\to\infty} \Phi_\infty\Big(\alpha^{-k}\, \Big(T_\alpha^{-k}\big(\iota_\alpha(\mathbf{z})\big) \cap \Lambda_{\alpha,0}\Big)\Big) = \mathcal{G}(\iota_\alpha(\mathbf{z}))\,. \qedhere
\]
\end{proof}

Now, we can conclude the proof of our last theorem.

\begin{proof}[{Proof of Theorem~\ref{srsthm}}]
This follows immediately from Theorem~\ref{t:24} and Proposition~\ref{p:SRS}.
\end{proof}

\section*{Acknowledgments}
We thank W.~Narkiewicz and A.~Schinzel for indicating an argument of J.~W\'ojcik that led to the proof of Lemma~\ref{measuremult}.
We are also indebted to B.~Conrad and F.~Halter-Koch for their help and to C.~van de Woestijne for many stimulating discussions.

\bibliographystyle{amsalpha}
\bibliography{padic}

\newcommand{\etalchar}[1]{$^{#1}$}
\providecommand{\bysame}{\leavevmode\hbox to3em{\hrulefill}\thinspace}
\providecommand{\MR}{\relax\ifhmode\unskip\space\fi MR }
\providecommand{\MRhref}[2]{%
  \href{http://www.ams.org/mathscinet-getitem?mr=#1}{#2}
}
\providecommand{\href}[2]{#2}
\begin{thebibliography}{KLSW99}

\bibitem[ABB{\etalchar{+}}05]{Akiyama-Borbeli-Brunotte-Pethoe-Thuswaldner:05}
S.~Akiyama, T.~Borb{\'e}ly, H.~Brunotte, A.~Peth{\H{o}}, and J.~M. Thuswaldner,
  \emph{Generalized radix representations and dynamical systems. {I}}, Acta
  Math. Hungar. \textbf{108} (2005), no.~3, 207--238.

\bibitem[ABBS08]{ABBS:08}
S.~Akiyama, G.~Barat, V.~Berth{\'e}, and A.~Siegel, \emph{Boundary of central
  tiles associated with {P}isot beta-numeration and purely periodic
  expansions}, Monatsh. Math. \textbf{155} (2008), no.~3-4, 377--419.

\bibitem[AFS08]{Akiyama-Frougny-Sakarovitch:07}
S.~Akiyama, C.~Frougny, and J.~Sakarovitch, \emph{Powers of rationals modulo 1
  and rational base number systems}, Israel J. Math. \textbf{168} (2008),
  53--91.

\bibitem[Aki02]{Akiyama:02}
S.~Akiyama, \emph{On the boundary of self affine tilings generated by {P}isot
  numbers}, J. Math. Soc. Japan \textbf{54} (2002), no.~2, 283--308.

\bibitem[And00]{Anderson:00}
D.~D. Anderson, \emph{G{CD} domains, {G}auss' lemma, and contents of
  polynomials}, Non-{N}oetherian commutative ring theory, Math. Appl., vol.
  520, Kluwer Acad. Publ., Dordrecht, 2000, pp.~1--31.

\bibitem[AS07]{Akiyama-Scheicher:07}
S.~Akiyama and K.~Scheicher, \emph{Symmetric shift radix systems and finite
  expansions}, Math. Pannon. \textbf{18} (2007), no.~1, 101--124.

\bibitem[AT04]{Akiyama-Thuswaldner:04}
S.~Akiyama and J.~M. Thuswaldner, \emph{A survey on topological properties of
  tiles related to number systems}, Geom. Dedicata \textbf{109} (2004),
  89--105.

\bibitem[BK06]{Barge-Kwapisz:06}
M.~Barge and J.~Kwapisz, \emph{Geometric theory of unimodular {P}isot
  substitutions}, Amer. J. Math. \textbf{128} (2006), no.~5, 1219--1282.

\bibitem[BS05]{Berthe-Siegel:05}
V.~Berth{\'e} and A.~Siegel, \emph{Tilings associated with beta-numeration and
  substitutions}, Integers \textbf{5} (2005), no.~3, A2, 46 pp.

\bibitem[BSS{\etalchar{+}}11]{BSSST:11}
V.~Berth{\'e}, A.~Siegel, W.~Steiner, P.~Surer, and J.~Thuswaldner,
  \emph{Fractal tiles associated with shift radix systems}, Adv. Math.
  \textbf{226} (2011), 139--175.

\bibitem[BW01]{Bandt-Wang:01}
C.~Bandt and Y.~Wang, \emph{Disk-like self-affine tiles in $\mathbb{R}^2$},
  Discrete Comput. Geom. \textbf{26} (2001), 591--601.

\bibitem[Cas67]{Cassels:67}
J.~W.~S. Cassels, \emph{Global fields}, Cassels, J. W. S. and Fr{\"o}hlich, A.
  (eds.), Algebraic Number Theory, Academic Press, London, 1967, pp.~42--84.

\bibitem[CCR96]{Cerveau-Conze-Raugi:96}
D.~Cerveau, J.-P. Conze, and A.~Raugi, \emph{Ensembles invariants pour un
  op\'erateur de transfert dans {${\bf R}^d$}}, Bol. Soc. Brasil. Mat. (N.S.)
  \textbf{27} (1996), no.~2, 161--186.

\bibitem[DKV00]{Duvall-Keesling-Vince:00}
P.~Duvall, J.~Keesling, and A.~Vince, \emph{The {H}ausdorff dimension of the
  boundary of a self-similar tile}, J. London Math. Soc. (2) \textbf{61}
  (2000), 748--760.

\bibitem[GH94]{Groechenig-Haas:94}
K.~Gr{\"o}chenig and A.~Haas, \emph{Self-similar lattice tilings}, J. Fourier
  Anal. Appl. \textbf{1} (1994), 131--170.

\bibitem[GHR99]{Grochenig-Haas-Raugi:99}
K.~Gr{\"o}chenig, A.~Haas, and A.~Raugi, \emph{Self-affine tilings with several
  tiles. {I}}, Appl. Comput. Harmon. Anal. \textbf{7} (1999), no.~2, 211--238.

\bibitem[Hen88]{Henrici:88}
P.~Henrici, \emph{Applied and computational complex analysis. {V}ol. 1}, Wiley
  Classics Library, John Wiley \& Sons Inc., New York, 1988, Power
  series---integration---conformal mapping---location of zeros, Reprint of the
  1974 original, A Wiley-Interscience Publication.

\bibitem[HR63]{Hewitt-Ross:63}
E.~Hewitt and K.~A. Ross, \emph{Abstract harmonic analysis. {V}ol. {I}:
  {S}tructure of topological groups. {I}ntegration theory, group
  representations}, Die Grundlehren der mathematischen Wissenschaften, Bd. 115,
  Academic Press Inc., Publishers, New York, 1963.

\bibitem[HR70]{Hewitt-Ross:70}
\bysame, \emph{Abstract harmonic analysis. {V}ol. {II}: {S}tructure and
  analysis for compact groups. {A}nalysis on locally compact {A}belian groups},
  Die Grundlehren der mathematischen Wissenschaften, Band 152, Springer-Verlag,
  New York, 1970.

\bibitem[Hut81]{Hutchinson:81}
J.~E. Hutchinson, \emph{Fractals and self-similarity}, Indiana Univ. Math. J.
  \textbf{30} (1981), 713--747.

\bibitem[IR06]{Ito-Rao:06}
S.~Ito and H.~Rao, \emph{Atomic surfaces, tilings and coincidence. {I}.
  {I}rreducible case}, Israel J. Math. \textbf{153} (2006), 129--155.

\bibitem[Ken92]{Kenyon:92}
R.~Kenyon, \emph{Self-replicating tilings}, Symbolic dynamics and its
  applications ({N}ew {H}aven, {CT}, 1991), Contemp. Math., vol. 135, Amer.
  Math. Soc., Providence, RI, 1992, pp.~239--263.

\bibitem[KL00]{Kirat-Lau:00}
I.~Kirat and K.-S. Lau, \emph{On the connectedness of self-affine tiles}, J.
  London Math. Soc. (2) \textbf{62} (2000), 291--304.

\bibitem[KLSW99]{Kenyon-Li-Strichartz-Wang:99}
R.~Kenyon, J.~Li, R.~Strichartz, and Y.~Wang, \emph{Geometry of self-affine
  tiles {I}{I}}, Indiana Univ. Math. J. \textbf{48} (1999), 25--42.

\bibitem[KS10]{Kenyon-Solomyak:10}
R.~Kenyon and B.~Solomyak, \emph{On the characterization of expansion maps for
  self-affine tilings}, Discrete Comput. Geom. \textbf{43} (2010), 577--593.

\bibitem[KS12]{Kalle-Steiner:12}
C.~Kalle and W.~Steiner, \emph{Beta-expansions, natural extensions and multiple
  tilings associated with {P}isot units}, Trans. Amer. Math. Soc. \textbf{364}
  (2012), no.~5, 2281--2318.

\bibitem[KV98]{Kenyon-Vershik:98}
R.~Kenyon and A.~Vershik, \emph{Arithmetic construction of sofic partitions of
  hyperbolic toral automorphisms}, Ergodic Theory Dynam. Systems \textbf{18}
  (1998), no.~2, 357--372.

\bibitem[LL07]{Leung-Lau:07}
K.-S. Leung and K.-S. Lau, \emph{Disklikeness of planar self-affine tiles},
  Trans. Amer. Math. Soc. \textbf{359} (2007), no.~7, 3337--3355.

\bibitem[LLR13]{Lai-Lau-Rao:13}
C.-K. Lai, K.-S. Lau, and H.~Rao, \emph{Spectral structure of digit sets of
  self-similar tiles on {${\Bbb R}^1$}}, Trans. Amer. Math. Soc. \textbf{365}
  (2013), no.~7, 3831--3850.

\bibitem[LR03]{Lau-Rao:03}
K.-S. Lau and H.~Rao, \emph{On one-dimensional self-similar tilings and
  {$pq$}-tiles}, Trans. Amer. Math. Soc. \textbf{355} (2003), no.~4,
  1401--1414.

\bibitem[LW96a]{Lagarias-Wang:96a}
J.~C. Lagarias and Y.~Wang, \emph{Integral self-affine tiles in $\mathbb{R}^n$
  {I}. standard and nonstandard digit sets}, J. London Math. Soc. \textbf{54}
  (1996), no.~2, 161--179.

\bibitem[LW96b]{Lagarias-Wang:96b}
\bysame, \emph{Self-affine tiles in $\mathbb{R}^n$}, Adv. Math. \textbf{121}
  (1996), 21--49.

\bibitem[LW96c]{Lagarias-Wang:96c}
\bysame, \emph{Tiling the line with translates of one tile}, Invent. Math.
  \textbf{124} (1996), no.~1-3, 341--365.

\bibitem[LW97]{Lagarias-Wang:97}
\bysame, \emph{Integral self-affine tiles in $\mathbb{R}^n$ {II}. lattice
  tilings}, J. Fourier Anal. Appl. \textbf{3} (1997), 83--102.

\bibitem[LW03]{Lagarias-Wang:03}
\bysame, \emph{Substitution {D}elone sets}, Discrete Comput. Geom. \textbf{29}
  (2003), no.~2, 175--209.

\bibitem[Mah68]{Mahler:68}
K.~Mahler, \emph{An unsolved problem on the powers of {$3/2$}}, J. Austral.
  Math. Soc. \textbf{8} (1968), 313--321.

\bibitem[Neu99]{Neukirch:99}
J.~Neukirch, \emph{Algebraic number theory}, Grundlehren der Mathematischen
  Wissenschaften, vol. 322, Springer-Verlag, Berlin, 1999.

\bibitem[Odl78]{Odlyzko:78}
A.~M. Odlyzko, \emph{Nonnegative digit sets in positional number systems},
  Proc. London Math. Soc. (3) \textbf{37} (1978), no.~2, 213--229.

\bibitem[OW91]{Odlyzko-Wilf:91}
A.~M. Odlyzko and H.~S. Wilf, \emph{Functional iteration and the {J}osephus
  problem}, Glasgow Math. J. \textbf{33} (1991), no.~2, 235--240.

\bibitem[Pr{\"u}32]{Prufer:32}
H.~Pr{\"u}fer, \emph{{Untersuchungen \"uber Teilbarkeitseigenschaften in
  K{\"o}rpern.}}, J. Reine Angew. Math. \textbf{168} (1932), 1--36.

\bibitem[Sch18]{Schur:18}
I.~Schur, \emph{{\"U}ber {P}otenzreihen, die im {I}nnern des {E}inheitskreises
  beschr\"ankt sind. {II}}, J. Reine Angew. Math. \textbf{148} (1918),
  122--145.

\bibitem[Sie03]{Siegel:03}
A.~Siegel, \emph{Repr\'{e}sentation des syst\`{e}mes dynamiques substitutifs
  non unimodulaires}, Ergodic Theory Dynam. Systems \textbf{23} (2003),
  1247--1273.

\bibitem[SW99]{Strichartz-Wang:98}
R.~Strichartz and Y.~Wang, \emph{Geometry of self-affine tiles {I}}, Indiana
  Univ. Math. J. \textbf{48} (1999), 1--23.

\bibitem[Tat67]{Tate:67}
J.~T. Tate, \emph{Fourier analysis in number fields and {H}ecke's
  zeta-functions}, Cassels, J. W. S. and Fr{\"o}hlich, A. (eds.), Algebraic
  Number Theory, Academic Press, London, 1967, pp.~305--347.

\bibitem[Thu89]{Thurston:89}
W.~P. Thurston, \emph{Groups, tilings and finite state automata}, {AMS}
  {C}olloquium lectures, 1989.

\bibitem[Vij40]{Vijayaraghavan:40}
T.~Vijayaraghavan, \emph{On the fractional parts of the powers of a number.
  {I}}, J. London Math. Soc. \textbf{15} (1940), 159--160.

\end{thebibliography}
\end{document}